\documentclass[reqno]{amsart}

\title[]{Monte Carlo Approximations of Time-Nonlocal Diffusions in Bounded Domains}

\author[I. Biočić]{Ivan Bio\v{c}i\'{c}}
\author[D.E. Cedeño-Girón]{Daniel E. Cedeño-Girón}
\author[A. Mijatović]{Aleksandar Mijatovi\'{c}}
\author[B. Toaldo]{Bruno Toaldo}
\keywords{Monte Carlo methods, time-nonlocal diffusion, anomalous diffusion, killed diffusions, Cauchy--Dirichlet problems}
\date{\today}
\subjclass[2020]{Primary 65M75; Secondary 65C05, 65M15, 35R11, 60K50}

\usepackage{amssymb, amsmath, amsfonts}
\usepackage{natbib}
\usepackage{lscape,comment,upgreek}
\usepackage{color}
\usepackage{mathtools}
\usepackage{dsfont}
\usepackage{mathrsfs}
\usepackage{bm}
\usepackage{graphicx}
\usepackage{wrapfig}
\usepackage{blindtext}
\usepackage{caption}
\usepackage{enumitem}
\usepackage{tikz}
\usetikzlibrary {mindmap}
\usetikzlibrary{arrows,decorations.markings,arrows.meta}
\usetikzlibrary{calc}

\usepackage[top=3cm, bottom=3cm, right=2.8cm,left=2.8cm]{geometry}

\usepackage{cancel}

\usepackage{soul}

\newcommand{\overbar}[1]{\mkern 1.5mu\overline{\mkern-1.5mu#1\mkern-1.5mu}\mkern 1.5mu}
\usepackage[pdftex,plainpages=false,colorlinks,hyperindex,bookmarksopen,linkcolor=red,citecolor=blue,urlcolor=blue]{hyperref}

\numberwithin{equation}{section}

\DeclareMathAlphabet{\mathpzc}{OT1}{pzc}{m}{it}

\DeclareMathOperator\supp{supp}	
\newcommand{\R}{\mathbb{R}}
\newcommand{\N}{\mathbb{N}}
\newcommand{\pr}{\mathds{P}} 
\newcommand{\ex}{\mathds{E}} 
    \newcommand{\var}{\mathds{V}\text{ar}} 

\newcommand{\1}{\mathds 1}

\newcommand{\DD}{\mathcal{D}}

\newcommand{\wt}{\widetilde}
\newcommand{\diam}{\textrm{diam}}
\newcommand{\dist}{\textrm{dist}}

\definecolor{revisiongreen}{rgb}{0,0.5,0}

\bibpunct{[}{]}{;}{n}{,}{,}
\usepackage[linesnumbered]{algorithm2e}
\RestyleAlgo{ruled}

\usepackage{bbm} 
\newtheorem{theorem}{Theorem}[section] 
\newtheorem{lemma}[theorem]{Lemma}

\theoremstyle{remark}
\newtheorem{remark}[theorem]{Remark}
\newtheorem{example}[theorem]{Example}
\numberwithin{equation}{section} 
\usepackage{authblk}

\newtheorem{assumption}{}

\newtheorem{property}{}

\usepackage{caption}
\usepackage{subcaption}
\usepackage{graphicx}
\usepackage[pagewise]{lineno}
\usepackage{float} 

\DeclareMathAlphabet{\mathpzc}{OT1}{pzc}{m}{it}

\allowdisplaybreaks

\begin{document}
	
	\maketitle
	
	\begin{abstract}
		We develop and analyze a Monte Carlo method for sampling killed anomalous
diffusions obtained by time-changing Brownian motion with drift by the inverse of
a subordinator. The method targets probabilistic representations of time-nonlocal,
including time-fractional, Cauchy--Dirichlet problems on bounded domains. Since
inverse subordinators can be sampled exactly in broad classes, while Brownian exit
times are generally unavailable in arbitrary domains, we approximate the killed
Brownian component by an Euler scheme with discrete boundary detection. We prove
a square-root weak error bound with explicit dependence on the Laplace exponent of
the subordinator, and derive mean-square and central limit results for the resulting
Monte Carlo estimator.  A numerical example in the disk and one in a high-dimensional anisotropic shell illustrate the theoretical rates, computation time, and the mesh-free character of the method.
	\end{abstract}
	
	
	    \section{Introduction}
\label{sec:introduction}

The aim of this paper is to approximate, at a fixed time $T>0$ and a point
$x\in D$, the solution to a time-nonlocal Cauchy--Dirichlet problem in a bounded
domain. The model example is the time-fractional problem
\begin{equation}
\label{eq:intro_frac_problem}
\begin{cases}
\partial_t^\alpha u(t,x)=\mathcal{G}u(t,x), & t>0,\ x\in D,\\
u(t,x)=0, & t>0,\ x\in D^c,\\
u(0,x)=f(x), & x\in D,
\end{cases}
\end{equation}
where $\alpha\in(0,1)$, $\partial_t^\alpha$ is the Caputo derivative, i.e.,
\[\partial_t^\alpha f(t)=\frac{1}{\Gamma(1-\alpha)}\int_0^t f'(s)(t-s)^{-\alpha}ds,\]
and
\[
        \mathcal{G}g(x)
        =
        \mu\cdot \nabla g(x)
        +
        \frac12\sum_{i,j=1}^d \Sigma_{ij}\partial_{ij}g(x)
\]
is the generator of a Brownian motion, say $X = (X_t)_{t \geq 0}$, with drift $\mu$ and covariance matrix
$\Sigma$. More general time-nonlocal equations are obtained by replacing the
Caputo derivative with convolution-type derivatives associated with subordinators;
this theory is developed, e.g., in
\cite{anh2016space, BLM-SpaceTime18,chen-chaos,chen2012space,dovidio_bounded,MeerschaertNaneVellaisamy2009,toaldo2015levy}, and the reader can consult \cite[Chapter 17]{schilling2026bernstein} for a recent review. Our method will also apply to this generalization.

The probabilistic representation of \eqref{eq:intro_frac_problem} or its more general version is given in
terms of a Brownian motion time-changed by the inverse of a subordinator. Let
$S=(S_t)_{t\geq 0}$ be a subordinator (i.e., a non-decreasing Lévy process), independent of $X$. Assume that $S$ is strictly increasing and let
\[
        L_t\coloneqq\inf\{s>0:S_s>t\}
\]
be its inverse. Define the time-changed process
\[
        \chi_t\coloneqq X_{L_t}, \qquad t\geq 0.
\]
The inverse time-change makes $\chi$ non-Markovian, although it admits a semi-Markov description after augmenting the state space by an age variable \cite{Meerschaert2014}, i.e., it can be embedded in a Markov process. Such processes naturally describe anomalous diffusion: their mean-square displacement may grow nonlinearly in time, giving subdiffusive, diffusive, or superdiffusive regimes depending on the subordinator's jumps and on the presence of drift.

Let
\[
        T_D=\inf\{t>0:X_t\notin D\},
        \qquad
        \tau_D=\inf\{t>0:\chi_t\notin D\},
\]
denote the exit times of processes $X$ and $\chi$, respectively, from the open set $D$.
The connection between the exit time of the time-changed process and the exit
time of the underlying Brownian motion gives, under the standing assumptions (see, e.g., \cite{ascione_annals}),
\[
        \{T<\tau_D\}=\{L_T<T_D\}.
\]
 Hence, the solution to be approximated is
\begin{equation}
\label{eq:intro_target}
        u(T,x)
        =
        \mathbb{E}_x\!\left[
            f(\chi_T)\mathds{1}_{\{T<\tau_D\}}
        \right]
        =
        \mathbb{E}_x\!\left[
            f(X_{L_T})\mathds{1}_{\{L_T<T_D\}}
        \right].
\end{equation}

The Monte Carlo algorithm used in this paper is the following. For one sample,
first generate $L_T$, using exact sampling methods for inverse subordinators when
available \cite{Jorge2023b,Jorge2023a}. Then build the random time grid
\[
        (t_0,\, t_1,\,t_2,\,\dots)=(0,h,2h,\dots,\lfloor L_T/h\rfloor h,L_T).
\]
Starting from $Y_0=x$, simulate the Brownian increments with drift along this
grid:
\[
        Y_{t_i}
        =
        Y_{t_{i-1}}
        +
        \mu(t_i-t_{i-1})
        +
        \sigma(W_{t_i}-W_{t_{i-1}}),
\]
where $\sigma$ is a square root of the covariance matrix $\Sigma$.
If one of the simulated points leaves $D$ before the terminal time $L_T$, the
sample is set equal to zero. Otherwise, the sample is set equal to $f(Y_{L_T})$.
Thus one realization returns
\[
        {Z_h
        =
        f(Y_{L_T})\mathds{1}_{\{L_T<\widehat{T}_{D}^{\,h,L_T}\}}},
\]
{where $\widehat{T}_{D}^{\,h,L_T}$ is the (endpoint-augmented) discrete exit time defined in Subsection~\ref{ss:scheme}.} Repeating this
procedure independently $N$ times gives the estimator
\begin{equation}
\label{eq:intro_mc_estimator}
        u_N^h(T,x)
        =
        \frac1N\sum_{k=1}^N Z_h^{(k)}.
\end{equation}

The typical advantage of Monte Carlo approach, compared with deterministic
space-time discretizations, is that it targets directly
the pointwise value $u(T,x)$. One does not need to construct a spatial mesh, store
the full time history of the nonlocal derivative, or assemble a global
finite-difference or finite-element operator. More precisely, deterministic numerical methods for time-nonlocal diffusion problems on bounded domains include finite-difference and spectral schemes for Caputo-type subdiffusion \cite{LiXu2009, LinXu2007}, convolution-quadrature and finite-element schemes \cite{CuestaLubichPalencia2006,JinLazarovZhou2013,JinLazarovZhou2016L1,JinLazarovZhou2016FullyDiscrete}, discontinuous-Galerkin and graded-mesh methods \cite{McLeanMustapha2009,Mustapha2011,StynesORiordanGracia2017}, and finite-difference/finite-element or Laplace-transform approaches for distributed-order models \cite{BuXiaoZeng2017,JinLazarovSheenZhou2016,YeLiuAnh2015}. A related deterministic approach for distributed-order time-fractional diffusion-wave equations, based on numerical inverse Laplace transforms and discontinuous Galerkin spatial discretization, was proposed by Engstr\"om, Giani and Grubi\v{s}i\'c \cite{EngstromGianiGrubisic2023}. The Monte Carlo method developed here is complementary to these deterministic solvers and may be preferable when only pointwise values of the solution are required. This is true especially in high dimension or in geometries where building a space-time mesh and handling the memory history are costly, since the time non-locality is reduced to sampling the inverse clock and the boundary condition is enforced by pathwise membership tests. We illustrate this in our Examples, in particular in Example \ref{ex:high-dimensional-shell} where $D$ is chosen to be an anisotropic high-dimensional shell.
Our method is also naturally parallelizable, since the
samples in \eqref{eq:intro_mc_estimator} are independent. 
Related Monte Carlo methods for fractional partial differential equations have been studied in
\cite{kolokoltsov2021}.

Once the algorithm has been fixed, the main mathematical question is to evaluate
the error of the approximation. There are two sources of error. The first is the
bias caused by replacing the true killed process
$f(X_{L_T})\mathds{1}_{\{L_T<T_D\}}$ by the discretely killed approximation
{$f(Y_{L_T})\mathds{1}_{\{L_T<\widehat{T}_{D}^{\,h,L_T}\}}$}. The second is the statistical error caused
by averaging only finitely many independent samples. The exact simulation of
$L_T$ is available for broad classes of subordinators, whereas exact simulation of
$T_D$ is only available in special geometries, such as balls, or in one-dimensional
settings \cite{HerrmannZucca2019,HerrmannZucca2020,Hsu1986}. Therefore the
essential approximation in the present paper is the discrete detection of the
Brownian exit time.

The analysis is inspired by Gobet's work on the weak approximation of killed
diffusions by Euler schemes \cite{Gobet2000}. However, our setting contains an
additional difficulty: the deterministic time horizon in the killed diffusion
estimate is replaced by the random horizon $L_T$. It is therefore not enough to
know that the killed Euler scheme has weak error of order $\sqrt{h}$ for each fixed
time. We need to track explicitly how the constants depend on the time horizon,
and then integrate those bounds with respect to the law of $L_T$.

For this reason, the present paper carries out the full error analysis for
Brownian motion with drift, that is, for constant coefficients in the underlying
SDE. This restriction is deliberate. For a general diffusion
\[
        \mathrm{d}X_t=b(X_t)\,\mathrm{d}t+\sigma(X_t)\,\mathrm{d}W_t,
\]
the Euler interpolation no longer has the same dynamics as the true process
between two grid points. In the decomposition of the weak approximation error
used in \cite{Gobet2000}, this produces additional interior error terms, in
addition to the boundary-crossing terms. If one tries to keep the dependence on
the time horizon explicit, these extra terms require cumbersome estimates on
derivatives of the associated parabolic problem, stochastic-flow terms, and
integration-by-parts arguments such as those based on Malliavin calculus. In the
Brownian-with-drift case, these complications disappear: the continuous
interpolation has the correct Brownian dynamics, and the dominant approximation
error comes from the fact that the discrete scheme may miss a boundary crossing.

The first main result is a deterministic-horizon weak error estimate for killed
Brownian motion with drift. Under standard regularity assumptions on the domain
and either a support condition or a boundary compatibility condition on $f$, we
prove that
\[
 \left|
 \mathbb{E}_x\!\left[
     f(X_s)\mathds{1}_{\{s<T_D\}}
     -
     {f(Y_s)\mathds{1}_{\{s<\widehat{T}_{D}^{\,h,s}\}}}
 \right]
 \right|
 \leq
 C(1+s)A(f)\sqrt{h},
 \qquad s>0,
\]
where $A(f)$ denotes the corresponding suitable norm of the function $f$. The important
point is the explicit linear dependence on $s$.

Conditioning on $L_T$ and using the moment bound
\[
        \mathbb{E}L_T\leq \frac{e}{\phi(1/T)}
\]
then gives the weak error bound for the killed time-changed process:
\[
 \left|
 \mathbb{E}_x\!\left[
     f(X_{L_T})\mathds{1}_{\{L_T<T_D\}}
     -
     {f(Y_{L_T})\mathds{1}_{\{L_T<\widehat{T}_{D}^{\,h,L_T}\}}}
 \right]
 \right|
 \leq
 C\left(1+\frac{e}{\phi(1/T)}\right)A(f)\sqrt{h}.
\]
Thus the error induced by the discrete killing rule is controlled explicitly in
terms of the Laplace exponent of the subordinator.

We also analyze the Monte Carlo estimator \eqref{eq:intro_mc_estimator}. We prove
that its mean-square error satisfies
\[
        \mathbb{E}_x\!\left[
             \bigl(u_N^h(T,x)-u(T,x)\bigr)^2
        \right]
        \leq
        \frac{\|f\|_\infty^2}{N}
        +
        C A(f)^2 h.
\]
The first term is the statistical error, while the second term is the squared weak
bias due to the discrete boundary detection. Finally, choosing $h=h_N$ so that
the bias is negligible at the central-limit scale, for instance $h_N=N^\delta$
with $\delta<-1$, yields a central limit theorem for the Monte Carlo estimator.

The paper is organized as follows. Section~\ref{s1} introduces the Brownian motion with
drift, inverse subordinators, killed time-changed processes, and the sampling
scheme. Section~\ref{s2} proves the deterministic-horizon killed Brownian error estimate,
then integrates it over the inverse-subordinator clock to obtain the main weak
error bound. The same section contains the mean-square and central limit analysis
of the Monte Carlo estimator. Finally, Section~\ref{s:examples} contains numerical examples, one in the disk and one in a high-dimensional anisotropic shell, which illustrate the theoretical rates and computation time of our method.

	\section{Preliminaries and assumptions}\label{s1}
Let $X=(X_t,\,t\geq0)$ denote the Brownian motion in $\R^d$, $d\ge1$, with the transition density of $X_t$ is given by
\begin{align}\label{transition}
    \pr_x(X_t\in dy)/dy = p(t,x,y) = \frac{\exp\left(-(y-x-\mu t)^T\Sigma^{-1}(y-x-\mu t)/2t\right)}{((2\pi t)^d\,\text{det}\Sigma )^{1/2}},
\end{align}
for all $t>0$, and $x,y\in \R^d$, where $\mu\in\R^d$, and ${\Sigma=(\Sigma_{ij})_{1\le i,j\le d}}\in \R^{d\times d}$ is a positive definite matrix. We decompose $\Sigma=\sigma\sigma^T$ for (any, but fixed) $\sigma\in\R^{d\times d}$ so that $X_t\overset{D}=\mu+\sigma W_t$, where $W=(W_t,\,t\geq0)$ is the standard Brownian motion in $ \R^d$. Denote by $\Sigma_-$ the smallest eigenvalue of $\Sigma^{-1}$ and by $\Sigma_+$ the largest eigenvalue of $\Sigma^{-1}$, so it holds
\begin{align}\label{eq: parabolic}
    \Sigma_-\,|y-x-\mu t|^2\leq (y-x-\mu t)^T\Sigma^{-1}(y-x-\mu t) \leq \Sigma_+\,|y-x-\mu t|^2,
\end{align}
for all $x,y\in \R^d$ and $t\ge0$, where $|\cdot|$ denotes the standard Euclidean norm in $\R^d$, the notation which we use throughout the paper.

   We will denote $\pr_x$ the probability measure such that $X_0=x$, $\pr_x$--a.s., and $\ex_x$ will be the corresponding expectation operator.
	Let $D\subset \R^d$ be a bounded open set, and let $T_D=\inf\{t>0:X_t\notin D\}$ be the first exit time of the Brownian motion $X$ from the set $D$. The killed Brownian motion upon exiting the set $D$ is denoted by $X^D$ and given by
	\begin{align}
		X_t^D=\begin{cases}
			X_t,&t<T_D,\\
			\partial, &t\ge T_D,
		\end{cases}
	\end{align}
	where $\partial$ is an additional point added to $\R^d$ called {\it the cemetery}. The killed process $X^D$ has a transition density $p_D(t,x,y)$ for which the Hunt formula holds:
	\begin{align}\label{hunt}
		\pr_x(X^D_t\in dy)/dy=p_D(t,x,y)=p(t,x,y)-\ex_x[p(t-T_D,X_{T_D},y)\1_{t>T_D}],
	\end{align}
	for all $x,y\in \R^d$, and $t>0$, see, e.g., \cite[Section 2.2, Eq. 4]{ChungZhao}. Moreover, $p_D$ is highly regular and it holds that
	\begin{align}\label{eq:1425}
	((0, +\infty) \times D \times D) 	\ni (t, x, y) \mapsto |\partial^\beta_xp_D(t,x,y)|\le \frac{c_0(d,D,\Sigma,\mu)}{t^{(d+|\beta|)/2}}e^{-c_1(\Sigma)\frac{|x-y|^2}{t}},
	\end{align}
    for each multi-index $\beta$ up to the integer order of the smoothness of the domain $D$. This seems to be a well-known result, see e.g. \cite{Garroni1992, Ladyzenskaja1968}, but it is usually stated for a fundamental solution to a general parabolic equation, and in such case  the constant $c_0$ may also depend on a time horizon $T \geq t$. Therefore, we will prove again \eqref{eq:1425} in Lemma \ref{ap:l:dens-reg} showing that in the Brownian motion case we do not have such additional time dependence.

It is well known that if $D$ is regular for the Brownian motion, i.e., $\pr_x (T_D =0)=1$, for every $x \in \partial D$, the Cauchy-Dirichlet problem 
\begin{alignat}{2}\label{bpde}
        (\partial_t + G)v(t,x)&=0, &\quad (t,x)&\in(0,T)\times D,\notag\\
    v(t,x)&=0, &\quad (t,x)&\in[0,T]\times D^c,\\
    v(T,x)&=f(x), &\quad x&\in D,\notag
\end{alignat}

where $G$ is the infinitesimal generator of $X$, i.e.
\begin{align}
    Gu(x) = \sum_{i=1}^d \mu_i\partial_{x_i}u(x) + \frac{1}{2} \sum_{i,j=1}^d {\Sigma_{ij}} \partial_{x_i x_j}^2u(x),
\end{align}
has a pointwise solution
\begin{align}\label{feynmann-kac}
    v(t,x)=\int_D p_D(T-t,x,y)f(y)dy =\ex_x[f(X_{T-t})\1_{T-t<T_D}].
\end{align}
Here, the classical theory of Feller operators covers the class of initial values $f\in \DD(G)=\{f\in C_0(D): Gf\in C_0(D)\}$, where $C_0(D)$ denotes continuous functions vanishing at the boundary of $D$. However, for $f\in L^2(D)$, the function in \eqref{feynmann-kac} is a weak/mild solution to \eqref{bpde}, which by parabolic regularization is smooth in $(0,T)\times D$, and consequently it also pointwisely solves \eqref{bpde}.
	
	To introduce killed subdiffusions, we will consider a time-changed Brownian motion killed upon exiting the domain $D$. To this end, let $S=(S_t,\,t\geq0)$ be a subordinator (i.e. a non-negative L\'evy process) with $S_0=0$, independent of the Brownian motion $X$, with the Laplace exponent 
	\begin{align}\label{eq:bernstein}
		\phi(\lambda)=b\lambda+\int_0^\infty(1-e^{-\lambda s})\nu(ds),\quad \lambda\ge 0.
	\end{align}
	Here $b\ge0$ and is called the drift of the subordinator, and $\nu$ is a measure such that $\int_0^\infty (1\wedge s)\nu(ds)<+\infty$ and is called the L\'evy measure of the subordinator. The function in \eqref{eq:bernstein} is called a Bernstein function, and such functions characterize subordinators, see \cite[Theorem 5.2]{bernstein}. For a rich collection of Bernstein functions refer to \cite[Chapter 16]{bernstein}.
    
    In this paper, we are always assuming that $b> 0$ or $\nu(0, +\infty) = +\infty$, i.e. $S$ is not a compound Poisson process, and it means that $S$ is strictly increasing.
    Consider now the inverse of $S$, i.e., the process $L=(L_t,\,t\geq0)$ where $L_t=\inf\{s>0:S_s>t\}$ (see more on inverse subordinators in \cite{bertoin1999,ascione2024regularity, rivero_doney}).

    Let $\chi=(\chi_t,\,t\geq0)$ be the process defined by $\chi_t=X_{L_t}$, $t\geq0$. This process is not Markovian since the time-change induces intervals of constancy with non-exponential distribution; however, it enjoys the so-called semi-Markov property, i.e. $((\chi_t, \gamma_t), t \geq 0)$, where $\gamma_t \coloneqq t-S_{L_t-}$ is a simple Markov process, see \cite[Theorem 4.1]{Meerschaert2014}. Furthermore, this process exhibits different diffusivity regimes. Indeed, without loss of generality, let $\chi_t$ start at zero a.s. By a simple conditioning argument, we have
    \begin{align}
        \ex_0 |\chi_t|^2 \, = \,  \ex_0 |X_{L_t}|^2 = \mu^T\mu \,\ex_0 L_t^2 + \textrm{Tr}(\Sigma)\, \ex_0 L_t,
    \end{align}
    where $\textrm{Tr}(\Sigma)$ denotes the trace of $\Sigma$.
    Then, for example, consider the case $S$ is an $\alpha$-stable subordinator with $\alpha\in(0,1)$. One has, from \cite[Eq. (3.18)]{kolokoltsov2021}, that $\ex L_t = c_1 t^\alpha$ and $\ex L_t^2 = c_2 t^{2\alpha}$, $c_1,\,c_2>0$. If $\mu\not\equiv0$, the process $\chi_t$ exhibits subdiffusive behavior for $\alpha\in(0,0.5)$, diffusive for $\alpha=0.5$, and superdiffusive for $\alpha\in(0.5,1)$, whereas, in the case $\mu=0$, the process $\chi_t$ is always subdiffusive.

    For a general subordinator $S$, the bounds for $\ex L_t^k$, $k\in\N$, for all $t>0$, are: $c\leq \phi(1/t) \ex L_t \leq e$, $0<c<e$, see \cite[Chapter III, Proposition 1]{bertoin1996}; and $c^k \leq \phi^k(1/t)\ex L_t^k \leq e\Gamma(1+k)$, $k\geq1$, where $\Gamma$ denotes the gamma function. The upper bound is provided in Lemma~\ref{lemma: L_bound} while the lower bound comes from the first moment's lower bound and Jensen's inequality.

Denote by $\chi^D$ the process $\chi$ killed upon exiting the set $D$, i.e.
\begin{align}
    \chi_t^D =
    \begin{cases}
        \chi_t, & t<\tau_D, \\
        \partial, & t\geq \tau_D,
    \end{cases}
\end{align}
where $\tau_D \coloneqq \inf\{t>0:\, \chi_t\notin D\}$ is the first exit time of the process $\chi$ from the open set $D$. Since $\chi$ is obtained by the time-change of $X$ with the inverse subordinator $L$, there is a strong connection between the exit times $T_D$ and $\tau_D$.
In particular, it holds that $\tau_D=S_{T_D-}$, see, e.g., \cite{ascione_annals}. However, since $S$ is independent of $X$ and since at any fixed time $t$, the probability that the process $S$ jumps is zero, by conditioning on $T_D$, we obtain $S_{T_D-}=S_{T_D}$ $\pr_x$-a.s., for all $x\in\R^d$. Hence, $\tau_D=S_{T_D}$ $\pr_x$-a.s., for all $x\in\R^d$.

With this at hand we define
\begin{align}\label{feynmann-kac_subordinated}
    u(T,x) \coloneqq \ex_x[f(\chi_T)\1_{T<\tau_D}]=\ex_x[f(\chi_T)\1_{L_T<T_D}],
\end{align}
where $T>0$ is a fixed time and $f$ is as in \eqref{bpde} (in our paper we either work under \ref{assumption_f} or \ref{assumption_f_}). This paper aims to study the Monte Carlo estimator of $u(T,x)$. In the subsequent subsection, we address the challenges associated with doing so.

We recall that the function $u(t,x)$ is the stochastic representation of the solution to a non-local (fractional-type) equation in the bounded domain $D$. Indeed, for $\phi(\lambda) = \lambda^\alpha$, i.e., the $\alpha$-stable subordinator case, the function \eqref{feynmann-kac_subordinated} satisfies the time-fractional equation
\begin{alignat}{2}
    \partial_t^\alpha u(t,x)  &= Gu(t,x), \quad &&t>0, x \in D, \label{eq1}\\
  u(t,x)  &= 0, \quad &&t>0,x \in D^c, \label{eq2}\\
  u(t,x) &= f(x),\quad && t=0, x \in D,\label{eq3}
\end{alignat}
where $\partial_t^\alpha f(t)=\frac{1}{\Gamma(1-\alpha)}\int_0^t f'(s)(t-s)^{-\alpha}ds$ is the fractional Caputo derivative, see, e.g. \cite{MeerschaertNaneVellaisamy2009} (when the Caputo derivative $\partial_t^\alpha$ and the generator $G$ are replaced by a more general time and space operators, see \cite{anh2016space, BLM-SpaceTime18,chen-chaos,chen2012space, dovidio_bounded,toaldo2015levy}). The classical well-posedness of the time-fractional
Cauchy-Dirichlet problem on bounded domains follows from
\cite[Theorem~3.1]{MeerschaertNaneVellaisamy2009}, where the unique classical
solution is represented by a killed Brownian motion time-changed by an inverse
stable subordinator; see also \cite[Theorem~3.6]{MeerschaertNaneVellaisamy2009}
for uniformly elliptic generators and \cite{Luchko2009Maximum,Luchko2010Uniqueness}
for maximum-principle-based uniqueness and existence results for generalized
time-fractional diffusion equations. For more general non-local-in-time equations, maximum-principle and
well-posedness results beyond the single Caputo kernel are available for
distributed-order and more general time-fractional operators; see, for instance,
\cite{Luchko2009DistributedOrder,LuchkoYamamoto2016}.

	\subsection{Sampling scheme}\label{ss:scheme}
The main task of the paper is to  pointwisely approximate, for fixed $x \in D$ and $T>0$, the function $u(T,x)$ defined in \eqref{feynmann-kac_subordinated}.

To accomplish this task, it is essential to sample the random variable $\chi_T\1_{T<\tau_D}=X_{L_T}\1_{L_T<T_D}$ (or approximate it). Clearly, since $X$ is a Brownian motion, one can get exact samples of $X$ at a fixed time. Regarding the inverse subordinator $L$, Algorithm 1 in \cite{Jorge2023b} enables one to generate exact samples of $L$ at a fixed time for a wide class of subordinators (see the exact condition on this class in \cite[Eq. (1.2) and Appendix A]{Jorge2023b}), while the algorithms in \cite{biocic2024}, for the same class of subordinators, enable the sampling of trajectories in the sense of finite-dimensional distributions $(L_{t_1}, \cdots, L_{t_n})$, for any choice of times $0\le t_1<\cdots<t_n$ and $n \in \mathbb{N}$. The exact sampling of the random variable $T_D$, however, is only feasible in a few cases, such as $d$-dimensional spheres and in dimension one \cite{HerrmannZucca2019, HerrmannZucca2020, Hsu1986}. In light of the paucity of formulae for generating $T_D$ in general higher-dimensional domains, we are going to use an Euler scheme as in \cite{Gobet2000} to approximate $T_D$.

Consider the Euler-Maruyama scheme:
\begin{align}\label{euler}
    Y_{t_{i}} = Y_{t_{i-1}} + \mu(t_{i}-t_{i-1}) + \sigma(W_{t_{i}}-W_{t_{i-1}}),\quad Y_0 = x,
\end{align}
where $t_0=0$, $t_i=ih$, $i\in\N$ and $h>0$, and where $W=(W_t,\,t\geq0)$ is the standard Brownian motion started at $0$ under $\pr_x$ for all $x$, while we treat $\pr_x(Y_0=x)=1$. For every deterministic horizon $q>0$ define the (endpoint-augmented) discrete exit time by
\begin{align}\label{eq:endpoint-exit}
    \widehat{T}_{D}^{\,h,q}\coloneqq
    \inf\Big\{t\in\{h,2h,,\dots,\lfloor q/h\rfloor h, q\}:\ Y_t\notin D\Big\},
\end{align}
with the convention $\inf\emptyset=+\infty$. Thus, the terminal point $q$ is always inspected, including when $q/h$ is not an integer. Therefore, {$f(Y_{L_T})\1_{L_T<\widehat{T}_{D}^{\,h,L_T}}$} gives an approximation of $f(X_{L_T})\1_{L_T<T_D}$, and we will study the Monte Carlo estimator of $u(T,x)$ in the form of $u_N^h(T)$:
\begin{align}\label{MCestimator}
    u_N^h(T,x) = \frac{1}{N}\sum_{k=1}^N Z_h^k,\quad {Z_h = f\left(Y_{L_T}\right)\1_{L_T<\widehat{T}_{D}^{\,h,L_T}}},
\end{align}
where the superscript $k$ denotes the $k$-th independent copy of $Z_h$. Here, the term $u_N^h(T)$ is written without the dependence on $x$ since this dependence is transferred to the underlying probability measure $\pr_x$.

For $t\in[t_{i-1},t_{i})$, we denote the interpolation of $Y_{t_i}$, as $Y_t$, which is given by
\begin{align}
    Y_t = Y_{t_{i-1}} + \mu(t-t_{i-1}) + \sigma(W_t-W_{t_{i-1}}).
    \label{interpolated}
\end{align}
This process is for theoretical purposes only, as the information of the interpolated values is not used in the evaluation of \eqref{feynmann-kac_subordinated} (or, more precisely, its approximation $u_N^h(T)$ under $\pr_x$).

Although one might be tempted to sample $\widehat{T}_{D}^{\,h,L_T}$ first and then compare it to $L_T$, the following approach, schematized in Algorithm \ref{alg: 1}, is more efficient since only the mesh points up to and including $L_T$ need to be generated.

\begin{algorithm}
    \caption{This routine returns the random variable {$f(Y_{L_T})\1_{L_T<\widehat{T}_{D}^{\,h,L_T}}$}, for fixed $T>0$ and $x\in D$}\label{alg: 1}
    \KwData{set $D$, function $f$, time-step $h$, position $Y_0=x\in D$, final time $T>0$}
    Generate $L_T$ (by, e.g., \cite{Jorge2023b})\\
    Create the time grid $(t_0,t_1,t_2,\dots,t_{n-1},t_n)=(0, h, 2h, \dots, \lfloor L_T /h \rfloor h, L_T )$.\\
    \For{$i=1$ \KwTo $n$}{
        $Y_{t_i} \gets Y_{t_{i-1}} + \mu(t_i-t_{i-1}) + \sigma(W_{t_i}-W_{t_{i-1}})$\\
        \eIf{$Y_{t_i}\notin D$}{
         {$f(Y_{L_T}) \1_{L_T<\widehat{T}_{D}^{\,h,L_T}} \gets 0$} and \textbf{break}
        }{ {$f(Y_{L_T}) \1_{L_T<\widehat{T}_{D}^{\,h,L_T}} \gets f(Y_{t_i})$}}
    } 
\end{algorithm}

\begin{remark}
    Algorithm \ref{alg: 1} can be readily implemented to sample from the finite-dimensional distributions of the killed time-changed process, {i.e., to sample the functional $f(Y_{L_{T_1}}, \cdots, Y_{L_{T_n}}) \mathds{1}_{\{L_{T_n}<\widehat{T}_{D}^{\,h,L_{T_n}}\}}$}. For instance, in step 1, the values $(L_{s_1},\dots,L_{s_n})$ by using \cite{biocic2024}; in step 2, {create} the time grid $$(t_1,\dots,t_n)=\texttt{Sort}\big( (0,h,2h,\dots,\lfloor L_{s_n} /h \rfloor h),\,(L_{s_1},\dots,L_{s_n})\big);$$ in steps 3-10, repeat the same procedure, but if $Y_{t_i}\notin D$,
    put $Y_{t_k}=\partial$ for all $k=i,i+1,\dots,n$; add an additional step that keeps the values $Y_{t_i}$ where $t_i\in \{L_{s_1},\dots,L_{s_n}\}$.
\end{remark}

In the following sections, we first study the error of the approximation {$f(Y_{L_T})\1_{L_T<\widehat{T}_{D}^{\,h,L_T}}$} of the expression $f(X_{L_T})\1_{L_T<T_D}$. Then, we conduct an error analysis  of the approximation $u_N^h(T,x)$ of $u(T,x)$.

\section{Sampling killed anomalous diffusion}\label{s2}
\subsection{Exact error bounds for sampling killed Brownian motion}
	The approach we use to determine the error of the approximation is inspired by \cite{Gobet2000} where a suitable class of killed diffusion processes was studied and where the corresponding solutions to the Cauchy-Dirichlet problem \eqref{bpde} were approximated. The method in \cite{Gobet2000} provides a bound for the approximation error which is dependent on the time horizon $T>0$. The main issue with that approach in our context is that the time horizon here is randomized because of the time-change (i.e. the horizon is $L_T$).
    
A simple argument, using the independence between the Brownian motion $X$ and the subordinator $S$ yields
\begin{align}\label{Error}
    \begin{split}
        &\ex_x\left[f\left(Y_{L_T}\right)\1_{L_T<\widehat{T}_{D}^{\,h,L_T}}-f\left(X_{L_T}\right)\1_{L_T<T_D}\right] \\ 
        &\quad= \int_0^\infty \ex_x\left[f\left(Y_s\right)\1_{s<\widehat{T}_{D}^{\,h,s}}-f\left(X_s\right)\1_{s<T_D}\right]\pr_x(L_T\in ds).
    \end{split}
\end{align}
Note that the integrand, i.e. {$\ex_x\left[f\left(Y_s\right)\1_{s<\widehat{T}_{D}^{\,h,s}}-f\left(X_s\right)\1_{s<T_D}\right]$}, corresponds to the error induced by approximating the killed Brownian motion with drift by using the Euler scheme, studied in \cite{Gobet2000}. {The only difference is that the length of the last interval of the grid, i.e., ending at time $s$, is at most $h$, and not necessarily equal to $h>0$, but, as we will show, this will be inessential for the error evaluation.} Here it is important to recall that $\pr_x$ is the probability measure under which $X_0=x$ and $Y_0=x$, and to stress that the sub-index $x$ in $\pr_x$ has no effect on the law of the inverse subordinator $L$.

To obtain the error bounds, we impose regularity assumptions on $D$ and $f$, as in \cite{Gobet2000}, where it is useful to recall the definition of domain classes from \cite[Section 6.2]{Gilbarg1977}.
From now on, we always assume:
\begin{assumption}\label{assumption_domain}
    The set $D$ is a bounded domain of class $C^{3+\alpha}$.
\end{assumption}
\noindent In \ref{assumption_domain} we assume a bit more than in \cite{Gobet2000} in order to explicitly track the dependence of the constants on the time horizon in the classical diffusion case.
Regarding the function $f$, we will use one of the two following assumptions.
\begin{assumption}\label{assumption_f}
    The function $f$ is a bounded measurable function, satisfying $\dist(\supp{f},\partial D)\geq 2\varepsilon>0$ for some $\varepsilon>0$.
 \end{assumption}

 \begin{assumption}\label{assumption_f_}
     The function $f\in C^{2+\alpha}(\overbar{D})$, for some $\alpha>0$, and satisfies $f(z)=Gf(z)=0$, $z\in\partial D$.
 \end{assumption}
 Here, $C^{k+\alpha}(\overbar{D})$ denotes the Banach space of bounded continuous functions in $D$ such that the derivatives of order $k\in\N$ are H\"{o}lder continuous in $\overbar{D}$ (with exponent $\alpha\in(0,1)$), where the norm is
\begin{align*}
    \|f\|_{k+\alpha,\,\overbar{D}} = \sum_{i=0}^k \sum_{|s|=i} \sup_{x\in D}|\partial_x^s f| + \sum_{|s|=k} \sup_{x,x^\prime\in D} \frac{|\partial_x^s f(x) - \partial_x^s f(x^\prime)|}{|x-x^\prime|^\alpha},
\end{align*}
and $s$ in the sums above denotes a multi-index.

\begin{lemma}
Assume \ref{assumption_domain}. On an event of $\pr_x$-probability one, for every $q>0$ and every sequence $h_n\to0$,
\begin{align}\label{eq:endpoint-exit-convergence}
    \widehat{T}_{D}^{\,h_n,q}\wedge q\longrightarrow T_D\wedge q.
\end{align}
Moreover, on $\{T_D\le q\}$ one has $\widehat{T}_{D}^{\,h_n,q}\to T_D$, and
\begin{align}\label{eq:endpoint-indicator-convergence}
    \1_{q<\widehat{T}_{D}^{\,h_n,q}}\longrightarrow \1_{q<T_D}.
\end{align}
Consequently, the same assertions hold with the random horizon $q=L_T$.
\end{lemma}
\begin{proof}
By \eqref{interpolated}, $Y_t=x+\mu t+\sigma W_t=X_t$ for all $t\ge0$. Under \ref{assumption_domain}, the uniform exterior cone condition implies that every boundary point is regular for the exterior of $D$. Hence
\[
    T_D=\inf\{t>0:X_t\notin\overline D\},\qquad \pr_x\text{-a.s.}
\]
Fix a path in this a.s. event, and for notational simplicity define the set of inspection points $\pi_{h_n}^q\coloneqq \{0,h_n,2h_n,\dots, \lfloor q /h_n\rfloor h_n\}\cup\{q\}$. If $T_D<q$, then for every $\rho>0$ there is an open interval contained in $(T_D,(T_D+\rho)\wedge q)$ on which $X$ lies outside $D$. Since the largest gap between consecutive points of $\pi_{h_n}^q$ is at most $h_n$, this interval must contain a point of $\pi_{h_n}^q$ for all sufficiently large $n$. Because $\widehat{T}_{D}^{\,h_n,q}\ge T_D$, it follows that $\widehat{T}_{D}^{\,h_n,q}\to T_D$. If $T_D=q$, the endpoint $q$ belongs to $\pi_{h_n}^q$ and $X_q\notin D$, so $\widehat{T}_{D}^{\,h_n,q}=q$. Finally, if $T_D>q$, every point of $\pi_{h_n}^q$ lies in $D$, so $\widehat{T}_{D}^{\,h_n,q}=+\infty$. These three cases prove \eqref{eq:endpoint-exit-convergence} and \eqref{eq:endpoint-indicator-convergence}. Since the event used above does not depend on $q$, one may take $q=L_T$ pathwise.
\end{proof}

		\begin{theorem}\label{Thm: Gobet}
			For a time $T>0$ and a time-step $h\in(0,1)$, let {$f(Y_T)\1_{T<\widehat{T}_{D}^{\,h,T}}$} be the numerical approximation to $f(X_T)\1_{T<T_D}$, as in Subsection~\ref{ss:scheme}. Then there exist a positive constant $C=C(d,D,\Sigma,\mu)$ and a function $K(T)\coloneqq C(T+1)$
        such that:
        \begin{enumerate}[label=(\alph*)]
            \item if  \ref{assumption_domain} and  \ref{assumption_f}, then
            \begin{align*}
				\Big|\ex_x\big[f(X_T)\1_{T<T_D}-{f(Y_T)\1_{T<\widehat{T}_{D}^{\,h,T}}}\big]\Big| \leq K(T) \frac{\| f\|_{\infty}}{1\wedge \varepsilon^2}\sqrt{h},
			\end{align*}
        \item  if \ref{assumption_domain} and \ref{assumption_f_} hold, then
        \begin{align*}
				\Big|\ex_x\big[f(X_T)\1_{T<T_D}-{f(Y_T)\1_{T<\widehat{T}_{D}^{\,h,T}}}\big]\Big| \leq K(T) \| f\|_{2+\alpha,\overbar{D}}\sqrt{h}.
			\end{align*}
        \end{enumerate}
		\end{theorem}
    
	\begin{proof}
		The general idea of this proof is essentially due to \cite{Gobet2000}. However, here we work just with the Brownian motion with drift in order to be able to track the explicit dependence of constants on the time horizon $T$ obtained in \cite{Gobet2000}. Hence, our proof is a bit simplified, but in the same time more delicate at some points.

        We concentrate on the proof of part~(a) under assumptions \ref{assumption_domain} and \ref{assumption_f}. Although part~(b) relies on the same stochastic decomposition, it cannot be obtained simply by replacing $\|f\|_\infty$ with $\|f\|_{2+\alpha,\overline D}$. In particular, the positive distance between $\supp(f)$ and $\partial D$ is used explicitly at several stages in the proof of part~(a).
        For this reason, we focus on part~(a), while indicating at relevant points how the argument should be adapted to cover part~(b).

        For notational convenience, throughout this proof set $\tau_h\coloneqq\widehat{T}_{D}^{\,h,T}$. Note that 
        \begin{align}
        \begin{split}\label{eq:1300}
            \ex_x\left[f(Y_T)\1_{T<{\tau_h}}\right]&=\ex_x\left[\ex_{Y_{T\wedge {\tau_h}}}[f(X_{T-T\wedge {\tau_h}})\1_{T-T\wedge {\tau_h}<T_D}]\1_{T<{\tau_h}}\right]\\
            &\hspace{-2em}=\ex_x\left[v(T\wedge {\tau_h},\,Y_{T\wedge {\tau_h}})\1_{T<{\tau_h}}\right]=\ex_x\left[v(T\wedge {\tau_h},\,Y_{T\wedge {\tau_h}})\right],
        \end{split}
        \end{align}
        where the second to last equality follows by the definition of $v$ in \eqref{feynmann-kac}, and the last line because $v$ is the solution to \eqref{bpde}. 
        Thus, we have
        \begin{align}\label{eq: error}
            \ex_x\left[f\left(Y_T\right)\1_{T<{\tau_h}}-f\left(X_T\right)\1_{T<T_D}\right] = \ex_x\left[v(T\wedge {\tau_h},\,Y_{T\wedge {\tau_h}})-v(0,Y_0)\right].
        \end{align}

As pointed out in \cite[Remark 2.1]{Gobet2000}, the spatial derivatives of $v$ have jumps at the boundary so the Brownian motion $(Y_t)_{0\le t< {\tau_h}}$ probably crosses $\partial D$. Therefore, although one might be tempted to apply the It\^o formula in \eqref{eq: error}, it cannot be applied directly. In other words, we need to tweak the process  $(Y_t)_{0\le t< {\tau_h}}$ a bit so we can use It\^o-like formula on it. This will be done by using \cite[Property 3.1]{Gobet2000} and the corresponding projection of processes to $\overline D$ defined therein. Before we bring the details, let us first rewrite \eqref{eq: error} a bit more.

Define the stopping time $T_R\coloneqq\inf\{t>0:\,Y_t\notin D(R)\}$, where $R>0$ and $D(R)=\{y \in \R^d:\dist(y,D)<R\}$. 
By a similar  argument as in \eqref{eq:1300}, we have
\begin{align}
    \ex_x&\left[f(Y_T)\1_{T<{\tau_h}}\right]=\ex_x\left[f(Y_T)\big(\1_{T<{\tau_h}\wedge T_R}+\1_{T_R<T<{\tau_h}}\big)\right]\nonumber\\
    &=\ex_x\left[v(T\wedge {\tau_h}\wedge T_R,\,Y_{T\wedge {\tau_h}\wedge T_R})\1_{T<{\tau_h}\wedge T_R}\right]+\ex_x\left[f(Y_T)\1_{T_R<T<{\tau_h}}\right]\nonumber\\
    &=\ex_x\left[v(T\wedge {\tau_h}\wedge T_R,\,Y_{T\wedge {\tau_h}\wedge T_R})\right]+\ex_x\left[f(Y_T)\1_{T_R<T<{\tau_h}}\right].\label{eq:1315}
\end{align}
By using \eqref{eq:1315}, and for fixed $\delta>0$ by adding and subtracting $v((T-\delta)\wedge {\tau_h}\wedge T_R,\,Y_{(T-\delta)\wedge {\tau_h}\wedge T_R})$ in \eqref{eq: error},  we get
\begin{align*}
     \ex_x\left[f\left(Y_T\right)\1_{T<{\tau_h}}-f\left(X_T\right)\1_{T<T_D}\right] = E_1 + E_2 + E_3,
\end{align*}
where
\begin{align*}
    E_1 &= \ex_x\left[ f(Y_T)\1_{T_R<T<{\tau_h}} \right],\\
    E_2 &= \ex_x\left[v(T\wedge {\tau_h}\wedge T_R,\,Y_{T\wedge {\tau_h}\wedge T_R}) -v((T-\delta)\wedge {\tau_h}\wedge T_R,\,Y_{(T-\delta)\wedge {\tau_h}\wedge T_R})\right],\\
    E_3 &= \ex_x\left[v((T-\delta)\wedge {\tau_h}\wedge T_R,\,Y_{(T-\delta)\wedge {\tau_h}\wedge T_R})-v(0,Y_0)\right].
\end{align*}
Notice that, $\lim_{\delta\to0}E_2=0$ because of the continuity of $Y_{\cdot\wedge {\tau_h}\wedge T_R}$ and the dominated convergence theorem.

For $E_1$, it is easy to see that
\begin{align*}
    |E_1| &= \left|\ex_x\left[ f(Y_T)\1_{T_R<T<{\tau_h}} \right]\right| \leq \|f\|_\infty \pr_x(T_R<T<{\tau_h}).
\end{align*}
Let $\mathcal I_h^T$ denote the collection of intervals $(a,b]$ determined by consecutive distinct points of the mesh $\pi_h^T=\{0,h,2h,\dots,\lfloor T/h\rfloor h\}\cup\{T\}$. Thus, every $(a,b]\in\mathcal I_h^T$ satisfies $0<b-a\le h$, and
\begin{align*}
    |\mathcal I_h^T|\le  \left\lfloor\frac{T}{h}\right\rfloor+1\le \frac{T}{h}+1.
\end{align*}
On the event $\{T_R<T<\tau_h\}$, all mesh values up to $T$, including $Y_T$, lie in $D$. Hence, if $T_R\in(a,b]$, for some $(a,b]\in\mathcal I_h^T$, then $Y_a\in D$ and the interpolated process travels a distance at least $R$ during $(a,b]$. Lemma~\ref{lemma 4.1} therefore gives
\begin{align}\label{E1}
    |E_1|
    &\leq \|f\|_\infty\sum_{(a,b]\in\mathcal I_h^T}
    \pr_x\left(\sup_{s\in(a,b]}|Y_a-Y_s|\ge R\right)\leq 2d\|f\|_\infty\sum_{(a,b]\in\mathcal I_h^T}
    \exp\left(4C|\mu|^2(b-a)-\frac{CR^2}{b-a}\right)\notag\\
    &\leq 2d\|f\|_\infty\left(\frac{T}{h}+1\right)
    \exp\left(4C|\mu|^2h-\frac{CR^2}{h}\right)\leq C(T+1)\|f\|_\infty\sqrt h.
\end{align}
In the last step we used $h\in(0,1)$ and the fact that $h^{-1}\exp(-CR^2/h)\le C_R\sqrt h$. Thus, after absorbing the localization radius $R$ into the constant, we obtain
\begin{align}\label{E1-new}
    |E_1| &\leq C(T+1)\|f\|_\infty\sqrt h.
\end{align}
Note that the term $E_1$ actually vanishes exponentially with respect to $h$, but since the final rate of the theorem is of order $\sqrt{h}$, we keep the estimate as in \eqref{E1-new}.
Also, under \ref{assumption_f_}, the same estimate \eqref{E1-new} holds with $\|f\|_\infty$ bounded by $\|f\|_{2+\alpha,\overbar D}$. 

We are left to deal with $E_3$. Here, we use the regularity of $D$, i.e. the assumption \ref{assumption_domain}, and evoke \cite[Property 3.1]{Gobet2000} (refer to \cite[page 381 -- 384]{Gilbarg1977} for details).
\begin{property} [{Property of $C^{3+\alpha}$ domains}]\label{property 3.1}
    Under \ref{assumption_domain}, the domain $D$ enjoys the following properties: There exists $R>0$ such that for $D$ and
    \begin{align}
    V_{\partial D}(R)&\coloneqq\{z\in\R^d:\,\dist(z,\partial D)\leq R\},\\
    D(R)&\coloneqq\{z\in\R^d:\,\dist(z,D)<R\},\label{D(R)}
\end{align}
    it holds that
    \begin{itemize}
        \item[(i)] (Local diffeomorphism). For all $s\in\partial D$, there are two open bounded sets $U^s$ and $V^s$, a {$C^{3+\alpha}$}-diffeomorphism $F^s$ from $U^s$ into $(-2R,2R)\times V^s$, such that
        \begin{align*}
            \qquad\qquad F^s:
            \begin{cases}
                U^s\subset\R^d \longrightarrow (-2R,2R)\times V^s\subset\R\times\R^{d-1}, \\
                x\mapsto(z_1,z)\coloneqq(z_1,z_2,\dots,z_d)\ \text{such that}\ x=g^s(z) + z_1n(g^s(z)),
            \end{cases}
        \end{align*}
        where $g^s$ is a mapping of $\partial D$ in a neighborhood of $s$. Denote by $G^s \coloneqq (F^s)^{-1}$.
        \item[(ii)] (Distance to $\partial D$). Let $s\in\partial D$. On $U^s$, the first coordinate of $F^s$, the function $F_1^s(\cdot)$, is the algebraic distance to $\partial D$; thus, it does not depend on $s$ and we denote it by $F_1$. In other words, $|F_1(x)|=\dist(x,\partial D)$ and $F_1(x)>0$ (resp. $F_1(x)<0$) if $x\in D\cap U^s$ (resp. $x\in \overbar{D}^c\cap U^s$). It is a {$C^{3+\alpha}$} function on $\cup_{s\in\partial D}U^s=V_{\partial D}(2R)$, which we extend into a $C_b^{{3+\alpha}}(\R^d, \R)$ function, with the conditions $F_1(\cdot)>0$ on $D$ and $F_1(\cdot)<0$ on $\overbar{D}^c$. Note that $\partial D=\{x\in\R^d:\, F_1(x)=0\}$.
        \item[(iii)] (Orthogonal projection on $\overbar{D}$). Let $s\in\partial D$. For $x\in U^s$, the orthogonal projection on $\overbar{D}$ of $x$ is uniquely defined by
        \begin{align}
            \text{Proj}_{\overbar{D}}(x) = G^s([F_1(x)]^+,F_2^s(x),\dots,F_d^s(x)).
            \label{projection}
        \end{align}
    \end{itemize}

    Since $\partial D$ is compact, one can find a finite collection of points $s_i$, ${1\leq i\leq k}$, in $\partial D$ (to which we associate $G^i, F^i, U^i$ and $V^i$, respectively) such that
    $V_{\partial D}(3R/2)\subset\bigcup_{1\leq i\leq k} U^i$.
    Let $U^0$ be an open set satisfying $\dist(\partial D, \overbar{U^0})>0$ and such that
    $D(3R/2)\subset \bigcup_{0\leq i\leq k} U^i$.
    We then build a partition of unity subordinate to the covering $U^i$, ${0\leq i\leq k}$, namely a family of non-negative $C_b^\infty$ functions $\phi^i$, ${0\leq i\leq k}$, such that $\supp(\phi^i)\subset U^i$ and
    $\sum_{i=0}^k \phi^i=1$ on $\overbar{D(R)}$.
    Up to this point, the functions $F^i$ (resp.\ $G^i$) are only defined on $U^i$ (resp.\ $(-2R,2R)\times V^i$), and we extend them to smooth functions on $\R^d$.
\end{property}

From now on, we fix $R>0$ such that Property \ref{property 3.1} holds. The Property \ref{property 3.1} implies two important consequences. First, by \cite[Proposition 3.1]{Gobet2000} every continuous semimartingale, say $\wt Z_t$, $t\ge0$, living in $\overline{D(R)}$ may be orthogonally projected on $\overline D$, by using the transformation $t\mapsto \text{Proj}_{\overbar{D}}\wt Z_{t}$, so that the projected process still remains a continuous semimartingale. Second, by \cite[Corollary 3.1]{Gobet2000}, the process $t\mapsto v(t,\text{Proj}_{\overbar{D}}\wt Z_{t})$ is a continuous semimartingale as well. Both of these new semimartingales have useful decompositions as we show below. 

We return to dealing with $E_3$, and use \cite[Proposition 3.1]{Gobet2000} on the semimartingale $(Y_{t\wedge T_R\wedge {\tau_h}},{t\ge0})$ to define a new semimartingale $Z=(Z_t,\, t\ge0)$ by $Z_t\coloneqq\text{Proj}_{\overbar{D}}Y_{t\wedge T_R\wedge {\tau_h}}$. In particular, \cite[Proposition 3.1]{Gobet2000} yields
\begin{align*}
    dZ_t = \1_{Y_t\in D}dY_t + \1_{Y_t\notin D}dY_t^{\partial D} + \frac{1}{2}n(Y_t)dL_t^0(F_1(Y)), \quad t\le T_R\wedge {\tau_h}.
\end{align*}
Here, $n(x)$ denotes the unit inward pointing normal at $x$, $L_t^0(F_1(Y))$ is the one-dimensional local time of the continuous semi-martingale $F_1(Y)$ at time $t$ and level $0$, and  $Y_t^{\partial D}$ is the continuous semi-martingale with $Y_0^{\partial D}=0$ such that
\begin{align}
    \begin{split}\label{eq:2252a}
        \1_{Y_t\notin D}&dY_{j,t}^{\partial D} = \1_{Y_t\notin D}\sum_{i=1}^k \phi^i(Y_t)\left( \sum_{l=2}^d \partial_{z_l}G_j^i(0,F_2^i(Y_y),\dots,F_d^i(Y_y)) dF_l^i(Y_t)\right. \\
        &\qquad\left. +\frac{1}{2}\sum_{l,m=2}^d \partial_{z_l,z_m}^2 G_j^i(0,F_2^i(Y_y),\dots,F_d^i(Y_y)) d\langle F_l^i(Y_.),F_m^i(Y_.) \rangle_t \right).
    \end{split}
\end{align}
Since $F^i$'s are $C^2$ functions by Property \ref{property 3.1}, by classical It\^o's formula, it also holds for $t\le T_R\wedge {\tau_h}$
\begin{align}
    &dF_l(Y_t) = \sum_{r=1}^d \partial_{x_r}F_l(Y_t) dY_t^r + \frac{1}{2}\sum_{r,r^\prime =1}^d \partial_{x_r,x_{r^\prime}}^2F_l(Y_t) d\langle Y_.^r,Y_.^{r^\prime} \rangle_t \label{eq:2252b}\\
    &\hspace{3.3em}=\sum_{r=1}^d \mu_r\partial_{x_r}F_l(Y_t)dt + \sum_{r=1}^d \partial_{x_r}F_l(Y_t)d(\sigma B_t)^r+\frac{1}{2}\sum_{r,r^\prime =1}^d \Sigma_{r,r^\prime}\partial_{x_r,x_{r^\prime}}^2F_l(Y_t) dt, \notag\\
    &d\langle F_l(Y_.),F_m(Y_.) \rangle_t = \nabla F_l(Y_t)^T\Sigma  \nabla F_m(Y_t)^T dt, \label{eq:2252c}
\end{align}
where $Y^r$ denotes the $r$-th component of $Y$.
In the lines above, the notation $\langle\cdot ,\cdot\rangle_t$ denotes the quadratic variation process.

We also use \cite[Corollary 3.1]{Gobet2000} on the solution $v$ to $\partial_tv+Gv=0$ and the projection $Z$ of $Y$, to get
\begin{align*}
    dv(t,Z_t) &= \1_{Y_t\notin D}\left( \sum_{l=1}^d{\partial_{x_l} v}(t,Z_t) dY_{l,t}^{\partial D} + \frac{1}{2} \sum_{l,m=1}^d \partial^2_{x_l x_m}v(t,Z_t) d\langle Y_{l,.}^{\partial D},Y_{m,.}^{\partial D}\rangle_t\right)\\
    &\quad+\frac{1}{2}\sum_{l=1}^d {\partial_{x_l} v}(t,Y_t) n_l(Y_t) dL_t^0(F(Y)).
\end{align*}

Thus, from the relation above, for $E_3$ we have
\begin{align}
    E_3 &= \ex_x\left[ \int_0^{(T-\delta)\wedge {\tau_h}\wedge T_R} \frac{1}{2}\sum_{l=1}^d {\partial_{x_l} v}(t,Y_t)\, n_l(Y_t) dL_t^0(F_1(Y)) \right]\notag\\
    &\qquad+ \ex_x\left[ \int_0^{(T-\delta)\wedge {\tau_h}\wedge T_R} \sum_{l=1}^d{\partial_{x_l} v}(t,Z_t) \1_{Y_t\notin D} dY_{l,t}^{\partial D} \right.\notag\\
    &\qquad\qquad\qquad\qquad\left. +\frac{1}{2}\sum_{l,m=1}^d{\partial_{x_l x_m} v}(t,Z_t) \1_{Y_t\notin D} d\langle Y_{l_{,.}}^{\partial D}, Y_{m_{,.}}^{\partial D} \rangle_t \right]\notag.
\end{align}
Let $\varphi(t)=a$ whenever $t\in(a,b]$ for an interval $(a,b]\in\mathcal I_h^T$. Thus $\varphi(t)$ is the left endpoint of the mesh interval containing $t$, and $0<t-\varphi(t)\le h$. Using this notation, and the decomposition above, we obtain
\begin{align}
    \begin{split}\label{1554}
    |E_3|&\leq \frac{d}{2}c_1\frac{\|f\|_\infty}{1\wedge\varepsilon} \ex_x\left[L_{T\wedge {\tau_h}}^0(F_1(Y))\right] \\
        &\, + \left(dc_1\frac{\|f\|_\infty}{1\wedge\varepsilon}C_1 + \frac{d^2}{2}c_2\frac{\|f\|_\infty}{1\wedge\varepsilon^2}C_2\right) \ex_x\left[ \int_0^T \1_{Y_{\varphi(t)}\in D}\1_{Y_t\notin D} dt \right].
    \end{split}
\end{align}
Here, the derivatives of $v$ were bounded by Lemma~\ref{lemma 3.1} with its constants  $c_1=C(1,d,D,\Sigma,\mu)$ and $c_2=C(2,d,D,\Sigma,\mu)$, so the first line is clear. The second line 
and the constants $C_1=C_1(d,D,\Sigma,\mu),$ and $C_2=C_2(d,D,\Sigma,\mu)$ are obtained from \eqref{eq:2252a}--\eqref{eq:2252c} by bounding the derivatives of $F^i$'s and $G$ (which are, in particular, $C^2$ functions depending only on the geometry of $D$).

Moreover, under \ref{assumption_f_}, the second assertion of Lemma~\ref{lemma 3.1} reads
\begin{align}\label{eq:global-boundary-schauder-v}
    \sup_{0\le s<T}\sup_{x\in\overbar D}
    \bigl(|\nabla v(s,x)|+|D^2v(s,x)|\bigr)
    \le C\|f\|_{2+\alpha,\overbar D}.
\end{align}
Consequently, in \eqref{1554} the coefficients involving the first and second derivatives of $v$ are bounded directly by $C\|f\|_{2+\alpha,\overbar D}$, without using a positive distance between $\supp(f)$ and $\partial D$.

The integral term in~\eqref{1554} equals to
\begin{align*}
    \ex_x\left[ \int_0^T \1_{Y_{\varphi(t)}\in D}\1_{Y_t\notin D} dt \right] = \int_0^T \pr_x(Y_{\varphi(t)}\in D,Y_t\notin D) dt,
\end{align*}
and the integrand can be bounded as follows. If $\varphi(t)=0$, then $\pr_x(Y_t\notin D)\leq 1 \leq \sqrt{\frac{h}{t}}$ for $t < h$. Otherwise, the Euler scheme must travel at least the distance $\dist(Y_{\varphi(t)},\partial D)$ within the time interval $(\varphi(t),t]$. This, together with the Markov inequality, implies
\begin{align*}
    \pr_x(Y_{\varphi(t)}\in D,Y_t\notin D) &\leq \ex_x\left[\1_{Y_{\varphi(t)}\in D}\pr_x\left(\sup_{s\in(\varphi(t),t]}|Y_{\varphi(t)}-Y_s|\geq \dist(Y_{\varphi(t)},\partial D)\right) \right]\\
    &\leq 2d\,e^{c_3|\mu|^2 h}\, \ex_x\left[\1_{Y_{\varphi(t)}\in D} e^{-c_3 \frac{\dist^2(Y_{\varphi(t)},\partial D)}{t-\varphi(t)}} \right],
\end{align*}
where the last inequality follows from Lemma \ref{lemma 4.1} with $c_3$ as its constants, and $t-\varphi(t)\leq h$. Further, by exploiting the fact that the law of the interpolated Euler scheme and the law of the drifted Brownian motion coincide, using Lemma \ref{lemma bounds transition density}, we get that there exist constants  $c_4(d,D,\Sigma,\mu)$ and $c_5(d,D,\Sigma,\mu)$ such that 
\begin{align}\label{1338integral}
    \pr_x(Y_{\varphi(t)}\in D,Y_t\notin D)
    &\leq \frac{c_4}{\varphi(t)^{d/2}} \int_D {e^{-c_5\frac{|x-y|^2}{\varphi(t)}}} e^{-c_5\frac{\dist^2(y,\partial D)}{t-\varphi(t)} }dy.
\end{align}
By using the flattening of the boundary since $D$ satisfies \ref{assumption_domain}, we get
\begin{align*}
    \pr_x(Y_{\varphi(t)}\in D,Y_t\notin D)&\le \frac{c_6}{\varphi(t)^{d/2}}\int_{\R^{d-1}} e^{-c_5\frac{|x-y|^2}{\varphi(t)}} dy \int_0^{+\infty} e^{-c_5\frac{|x_1-y_1|^2}{\varphi(t)}-c_5\frac{y_1^2}{t-\varphi(t)}} dy_1\\
    &\le \frac{c_7}{\sqrt{\varphi(t)}}\int_{-\infty}^{+\infty} e^{-c_5\frac{|x_1-y_1|^2}{\varphi(t)}-c_5\frac{y_1^2}{t-\varphi(t)}} dy_1\le c_8\sqrt{\frac{t-\varphi(t)}{t}},
\end{align*}
where $c_8=c_8(d,D,\Sigma,\mu)$.
Since $t-\varphi(t)\leq h$, we conclude that there is a constant $c=c(d,D,\Sigma,\mu)$ such that
\begin{align}\label{boundE3second}
    \ex_x\left[ \int_0^T \1_{Y_{\varphi(t)}\in D}\1_{Y_t\notin D} dt \right]
    &\leq c\sqrt{T} \sqrt{h}.
\end{align}

The last term to be addressed in \eqref{1554} is the one involving the local time. Note that Tanaka's formula \cite[p. 222]{revuz1999} for $F_1(Y_{T\wedge {\tau_h}})$ yields
\begin{align*}
    \frac{1}{2}L_{T\wedge {\tau_h}}^0(F_1(Y)) &= (F_1(Y_{T\wedge {\tau_h}}))^- - (F_1(x))^- + \int_0^{T\wedge {\tau_h}}\1_{F_1(Y_t)\leq0} dF_1(Y_t)\\
    &= (F_1(Y_{T\wedge {\tau_h}}))^- + \int_0^{T\wedge {\tau_h}}\1_{F_1(Y_t)\leq0} dF_1(Y_t),
\end{align*}
 as $(F_1(x))^-=0$ since $x\in D$, where we write $F_1(\cdot)^-\coloneqq-\min(F_1(\cdot),0)$.

Applying the It\^o formula to $F_1(Y_t)$, along with the fact that $F_1(Y_t)$'s derivatives are finite by Property \ref{property 3.1}, enables us to write
\begin{align}\label{eq:2329}
    \frac{1}{2}\ex_x\left[ L_{T\wedge {\tau_h}}^0(F_1(Y)) \right] &= \ex_x\left[ (F_1(Y_{T\wedge {\tau_h}}))^- \right] + C\ex_x\left[\int_0^{T\wedge {\tau_h}}\1_{F_1(Y_t)\leq0} dt \right],
\end{align}
where
$C=C(d,D,\Sigma,\mu)$.
Also, note that:
\begin{align*}
    \ex_x\left[\int_0^{T\wedge {\tau_h}}\1_{F_1(Y_t)\leq0} dt \right] &\leq \ex_x\left[\int_0^T \1_{F_1(Y_t)\leq0}\1_{F_1(Y_{\varphi(t)})>0} dt \right]\\
    &\qquad= \ex_x\left[ \int_0^T \1_{Y_{\varphi(t)}\in D}\1_{Y_t\notin D} dt \right]\le c\sqrt{T} \sqrt{h},
\end{align*}
as in \eqref{boundE3second}, so the second term in \eqref{eq:2329} is appropriately bounded.
Regarding the first term in \eqref{eq:2329}, we can proceed in the same spirit as in \cite[Eqs. (63)--(65)]{Gobet2000}. Indeed, for each $(a,b]\in\mathcal I_h^T$, put $u=b-a$, so that $0<u\le h$. Since the endpoint $T$ belongs to $\pi_h^T$, if no discrete exit occurs by time $T$, then $Y_T\in D$ and $(F_1(Y_T))^-=0$. Therefore,
\begin{align}\label{eq:63}
    \ex_x\left[ (F_1(Y_{T\wedge\tau_h}))^- \right]
    &= \sum_{(a,b]\in\mathcal I_h^T}
    \ex_x\left[\1_{\{b=\tau_h\}}(F_1(Y_b))^-\right].
\end{align}
Moreover, $\{b=\tau_h\}=\{a<\tau_h\}\cap\{Y_b\notin D\}$, and the Markov property gives
\begin{align}\label{eq: 64}
    \ex_x\left[\1_{\{b=\tau_h\}}(F_1(Y_b))^-\right]
    = \ex_x\left[\1_{\{a<\tau_h\}}\ex_{Y_a}\left[(F_1(Y_u))^-\right]\right].
\end{align}
Since $F_1(Y_{T_D})=0$, the It\^o estimate for the increments of $F_1(Y)$ and $\ex|B_t-B_s|=\sqrt{2(t-s)/\pi}$ yield, uniformly in $z\in\overbar D$,
\begin{align*}
    \ex_z\left[(F_1(Y_u))^-\right]
    &=\ex_z\left[\1_{\{T_D<u\}}\bigl((F_1(Y_u))^--(F_1(Y_{T_D}))^-\bigr)\right]\\
    &\le C\sqrt u\,\pr_z(T_D<u).
\end{align*}
Applying Lemma~\ref{lemma 5.1} and using $u\le h$ gives
\begin{align}\label{eq: constant_R}
    \ex_z\left[(F_1(Y_u))^-\right]
    \le C\sqrt h\,\pr_z(Y_u\notin D).
\end{align}
Substituting this estimate in \eqref{eq: 64}, we obtain
\begin{align*}
    \ex_x\left[\1_{\{b=\tau_h\}}(F_1(Y_b))^-\right]
    &\le C\sqrt h\,\pr_x(a<\tau_h,\,Y_b\notin D)\\
    &=C\sqrt h\,\pr_x(b=\tau_h).
\end{align*}
Summing over all consecutive mesh intervals, including the possibly shorter terminal interval, yields
\begin{align}
    \ex_x\left[ (F_1(Y_{T\wedge\tau_h}))^- \right]
    &\le C\sqrt h\sum_{(a,b]\in\mathcal I_h^T}\pr_x(b=\tau_h)\\
    &=C\sqrt h\,\pr_x(\tau_h\le T)\le C\sqrt h.
\end{align}
Consequently,
\begin{align}\label{eq:1455}
     \frac{1}{2}\ex_x\left[ L_{T\wedge\tau_h}^0(F_1(Y)) \right]
     &\le C(\sqrt{T}\sqrt{h}+\sqrt{h}).
\end{align}

Therefore, by implementing the bounds \eqref{boundE3second} and \eqref{eq:1455} into \eqref{1554}, we get
\begin{align}\label{E3}
    |E_3| &\leq \wt c_1\frac{\|f\|_\infty}{1\wedge\varepsilon}(\sqrt{T}\sqrt{h}+\sqrt{h}) + \wt c_2\frac{\|f\|_\infty}{1\wedge\varepsilon^2}\sqrt{T}\sqrt{h},
\end{align}
for $\wt c_1=\wt c_1(d,D,\Sigma,\mu)$ and $\wt c_2=\wt c_2(d,D,\Sigma,\mu)$.

Under \ref{assumption_f_}, using \eqref{eq:global-boundary-schauder-v} in place of the support-dependent derivative bounds, gives by the same occupation-time and local-time estimates,
\begin{align*}
    |E_3|\le C(\sqrt T+1)\|f\|_{2+\alpha,\overbar D}\sqrt h.
\end{align*}

 {Combining \eqref{E1-new} and \eqref{E3} proves part~(a). Combining \eqref{E1-new}, \eqref{eq:global-boundary-schauder-v}, and the preceding bound for $E_3$ proves part~(b).}

    \end{proof}

\begin{remark}
      When the domain is convex, the orthogonal projection, which is one of the main ingredients of the proof of the previous theorem, is well-defined on the whole space, i.e. we may take $R=+\infty$. This means that, in the case of the convex $D$, the term $E_1$ is irrelevant, and other computations inside the proof simplify a bit. We note that the constant of Theorem \ref{Thm: Gobet} can be tracked almost perfectly form line to line, since almost all of them include classical computations with Gaussian density, or some elementary observations. However, for obtaining a truly explicit constant, one also needs to know some geometric quantities of $D$, e.g. the step from \eqref{1338integral} to \eqref{boundE3second} which includes the flattening of the boundary and highly connected Lemma \ref{lemma 5.1} which uses the exterior cone condition property. Also, more delicately, obtaining the step \eqref{1554} includes rather general constant $K(d,\Sigma)$ from \cite[Chapter 3, Theorem 5]{Friedman1964} (see Lemma \ref{ap:l:dens-reg}). In other words, these steps rely on the geometric properties of $D$, hence the constants are not entirely explicit even for a (general) convex $D$.

\end{remark}

	\subsection{Sampling killed anomalous diffusion}
    With the explicit dependence on $T$ of the bound provided by Theorem \ref{Thm: Gobet}, we can evaluate the error induced by sampling {$f(Y_{L_T})\1_{L_T<\widehat{T}_{D}^{\,h,L_T}}$} with Algorithm \ref{alg: 1}.
    
	\begin{theorem}\label{Thm: our}
        {Assume~\ref{assumption_domain}. Let $L=(L_t,\,t\geq0)$ be the inverse of a subordinator with its Laplace exponent $\phi$ as in \eqref{eq:bernstein}. For a fixed time $T>0$ and a time-step $h\in(0,1)$, let {$f(Y_{L_T})\1_{L_T<\widehat{T}_{D}^{\,h,L_T}}$} be the numerical approximation to $f(X_{L_T})\1_{L_T<T_D}$, constructed as in Subsection~\ref{ss:scheme}. Define $K_\phi(T)\coloneqq C( 1+ e/\phi(1/T) )$, where $C=C(d,D,\Sigma,\mu)$ is the positive constant appearing in Theorem \ref{Thm: Gobet}.
        \begin{enumerate}[label=(\alph*)]
            \item Under \ref{assumption_f}, it holds that
            \begin{align*}
                \Big|\ex_x\big[f(X_{L_T})\1_{L_T<T_D}-{f(Y_{L_T})\1_{L_T<\widehat{T}_{D}^{\,h,L_T}}}\big]\Big| \leq K_\phi(T) \frac{\| f\|_{\infty}}{1\wedge \varepsilon^2}\sqrt{h}.
		      \end{align*}
        \item  Under \ref{assumption_f_}, it holds that
        \begin{align*}
			\Big|\ex_x\big[f(X_{L_T})\1_{L_T<T_D}-{f(Y_{L_T})\1_{L_T<\widehat{T}_{D}^{\,h,L_T}}}\big]\Big| \leq K_\phi(T) \| f\|_{2+\alpha,\overbar{D}}\sqrt{h}.
		\end{align*}
        \end{enumerate}
        }
	\end{theorem}
	\begin{proof}
		In light of \eqref{Error} and Theorem \ref{Thm: Gobet}, under e.g. \ref{assumption_f}, it follows  that 
        \begin{align*}
            \left|\ex_x  \left[{f\left(Y_{L_T}\right)\1_{L_T<\widehat{T}_{D}^{\,h,L_T}}} -f\left(X_{L_T}\right)\1_{L_T<T_D}\right] \right| \leq C(\ex L_T+1)\frac{\| f\|_{\infty}}{1\wedge \varepsilon^2} \sqrt{h},
        \end{align*}
        for all $T>0$ and $h\in (0,1)$. Further, by \cite[Chapter III, Proposition 1]{bertoin1996} or Lemma~\ref{lemma: L_bound} below, we also have
        \begin{align*}
            \ex L_T \leq e/\phi(1/T).
        \end{align*}
	\end{proof}

	\subsection{Monte-Carlo analysis and statistical error}

    In this section, we provide an upper bound for the $L^2$ error of a Monte Carlo estimator of the function $u(t,x)$, defined in \eqref{feynmann-kac_subordinated}. Here, the Monte Carlo estimator of $u(t,x)$ is denoted by $u_N^h(t)$ and given by
    \begin{align}\label{MCestimator_}
        u_N^h(t) \coloneqq \frac{1}{N}\sum_{k=1}^N Z_h^k,
    \end{align}
    where the superscript $k$ denotes the $k$-th independent copy of the approximation ${Z_h \coloneqq f\left(Y_{L_t}\right)\1_{L_t<\widehat{T}_{D}^{\,h,L_t}}}$. In $u_N^h(t)$, we recall that the dependence on $x$ is hidden in the underlying probability measure $\mathds{P}_x$. In the calculations that follow, the theoretical value is denoted by $Z \coloneqq f(\chi_T)\1_{T<\tau_D}=f\left(X_{L_t}\right)\1_{L_t<T_D}$.  Moreover, throughout the section, we always assume that the domain $D$ satisfies \ref{assumption_domain} and that the function $f$ satisfies either \ref{assumption_f} or \ref{assumption_f_}.
    
    First we show, in the theorem below, that the approximation error vanishes in $L^2$ as $h\to0$ and $N\to+\infty$ jointly. Subsequently, we proceed to consider $h$ as a function of $N$ and prove the central limit theorem.

    \begin{theorem}\label{thm: l2error}
        There exist a positive constant $\mathcal{C}=\mathcal{C}(d,D,\Sigma,\mu,\phi)$ such that for all $x\in D$, $h\in(0,1)$, and $N\in \N$, it holds that:
        \begin{enumerate}[label=(\alph*)]
            \item under \ref{assumption_domain} and if $f$ satisfies \ref{assumption_f}, then
            \begin{align*}
			\ex_x\left(u_N^h(t)-u(t,x)\right)^2 \leq \frac{\|f\|^2_\infty }{N} + \mathcal{C}  \frac{\|f\|^2_\infty}{1\wedge \varepsilon^4} h;
		\end{align*}
        \item  under \ref{assumption_domain} and if $f$ satisfies \ref{assumption_f_}, then
        \begin{align*}
			\ex_x\left(u_N^h(t)-u(t,x)\right)^2 \leq \frac{\|f\|^2_\infty }{N} + \mathcal{C}  \|f\|^2_{2+\alpha,\overline D}\, h.
		\end{align*}
        \end{enumerate}
    \end{theorem}

    \begin{proof}
       {Since {$Z_h = f\left(Y_{L_t}\right)\1_{L_t<\widehat{T}_{D}^{\,h,L_t}}$} and $Z = f\left(X_{L_t}\right)\1_{L_t<T_D}$,  we have $\ex_x Z_h^2\le \|f\|^2_\infty<+\infty$, as well as, $\ex_x Z^2\le \|f\|^2_\infty < +\infty$.} 
        Therefore, by elementary manipulations,   
        \begin{align*}
            \ex_x\left(u_N^h(t)-u(t,x)\right)^2 = \frac{1}{N}\var_x Z_h + \left(\ex_x(Z-Z_h)\right)^2.
        \end{align*}
        {The first term above is trivially bounded by $\frac{\|f\|^2_\infty}{N}$, while 
        Theorem \ref{Thm: our} provides an upper bound for the second term.}
\end{proof}

Theorem~\ref{Thm: our} implies the following central limit theorem for the error of the Monte Carlo estimator $u_N^h(t)$ of $u(t,x)$.

\begin{theorem}\label{CLT}
    Assume \ref{assumption_domain} and that $f$ satisfies either \ref{assumption_f} or \ref{assumption_f_}, and set $h_N:=N^\delta$ for $\delta<-1$. Let
    \begin{align*}
        S_N\coloneqq\frac{\sqrt{N}\left(u_N^{{h_N}}(t)-u(t,x)\right)}{\sigma(t,x)},
    \end{align*}
    where $\sigma(t,x)^2=\var_x Z>0$. Then, for any $\psi\in C_b\left(\R\right)$, it is true that
    \begin{align*}
        \ex_x \psi(S_N) \longrightarrow\int_{\R} \psi(w) \frac{\exp(-w^2/2)}{\sqrt{2\pi}}dw,\quad\text{as}\quad N\to\infty.
    \end{align*}
\end{theorem}

\begin{proof}Recall
$Z=f(X_{L_t})\mathbf 1_{\{L_t<T_D\}}$ and
{$Z_h=f(Y_{L_t})\mathbf 1_{\{L_t<\widehat{T}_{D}^{\,h,L_t}\}}$}.
For each \(j=1,\dots,N\), let \(Z_j\) be the \(j\)-th independent copy of \(Z\), and let
\(Z_h^j\) be the \(j\)-th independent copy of \(Z_h\) and assume that $Z_h^j$ and $Z_j$ are coupled by via the same Brownain path. Then the Monte Carlo estimator
$u_N^h:=\frac1N\sum_{j=1}^N Z_h^j$
of
$u:=\mathds E Z$ satisfies the following for any
sequence \(h_N\downarrow0\):
\begin{equation}
\label{eq:CLT}
\sqrt N\bigl(u_N^{h_N}-u\bigr)
=
\frac1{\sqrt N}\sum_{j=1}^N (Z_j-\mathds EZ)
+A_N+B_N,
\end{equation}
where $A_N:=\frac1{\sqrt N}\sum_{j=1}^N
\Bigl[
(Z_{h_N}^j-Z_j)-\mathds E(Z_{h_N}-Z)
\Bigr]$ and $B_N:=
\sqrt N\bigl(\mathds EZ_{h_N}-\mathds EZ\bigr)$. Since $h_N=N^{\delta}$ for  $\delta<-1$, by Theorem~\ref{thm: l2error}, the bias satisfies  $\lim_{N\to\infty}B_N=0$.

{By the endpoint-grid convergence in \eqref{eq:endpoint-indicator-convergence}, applied pathwise with $q=L_t$, and by $Y=X$ in \eqref{interpolated}, one has $Z-Z_{h_N}\to0$ almost surely.} Since the function $f$ is bounded, $Z-Z_{h_N}\to0$ in $L^2$ as $N\to\infty$. Note $\ex A_N =0$ and $\var A_N =\var (Z-Z_{h_N})\leq \ex[(Z-Z_{h_N})^2]\to0$. Thus $A_N\to0$ in probability as $N\to\infty$.

Since the first summand on the right-hand side of~\eqref{eq:CLT} converges weakly to a normal distribution and $A_N+B_N\to 0$ in probability, Slutsky's theorem completes the proof.
\end{proof}

\section{Examples}\label{s:examples}
In this section, we give two detailed examples of our method. The first one deals with a time-fractional Cauchy-Dirichlet problem for the Laplacian in a ball. In it, we back up our claims in Theorem \ref{thm: l2error} and Theorem \ref{CLT} with numerical evidence. The second illustrates the applicability of our method in high-dimensions and in non-trivial geometry.

The simulations are performed in Python, using the standard libraries, e.g., \texttt{SciPy} library is used to invoke Bessel functions, while for the Mittag-Leffler function used in the examples, we use the \texttt{pymittagleffler} library, which was developed based on the work of \cite{Garrappa2015}.

All simulations were done on a Microsoft Surface Pro 8 with Intel i7-1185G7 CPU and 16 GB of RAM (no GPU was used). We attach the GitHub repository \cite{Python} containing the corresponding Python codes used to generate the figures in the following examples for the reader's convenience.
\begin{example}\label{ex:1} 
    Set $D=\{x\in\R^2:\, |x|<R\}$, $R>0$. Let $L=(L_t,t\ge0)$ be the inverse of an $\alpha$-stable subordinator, $\alpha\in (0,1)$, and $B=(B_t,t\ge0)$ the Brownian motion in $\R^2$, with generator $\Delta$, independent of $L$. In this scenario, the solution to the time-changed problem \eqref{eq1}--\eqref{eq3} reads
    \begin{equation}\label{Ex.Sol.}
        \begin{aligned}
            &u(t,r) = \sum_{n=1}^\infty c_n E_{\alpha}(-\lambda_n^2 t^\alpha) J_0\left(\lambda_n r \right),\\
            &c_n = \frac{2}{R^2 J_1(j_{0,n})^2}\int_0^R yf(y) J_0\left(\lambda_n y \right) dy,\quad\quad \lambda_n = \frac{j_{0,n}}{R},
        \end{aligned}    
    \end{equation}
    where $E_\alpha$ is the Mittag-Leffler function, $J_i(\cdot)$, $i\in\{0,1\}$, are the Bessel functions of the first kind, and $j_{0,n}$, $n\in\N$, is the sequence of positive zeros of $J_0$. For further details see \cite[Chapter 7]{Carslaw1921}.
    
    Although $u(t,x)$, as given in \eqref{Ex.Sol.}, cannot be strictly implemented, it can be considered (practically) exact by expanding the sum until a desired accuracy is reached, e.g., 15 significant digits. We apply this approach for the initial datum $f(y)=(1-y^2/R^2)^3$, where the corresponding $c_n$'s are derived in Appendix \ref{ap:B}. Further, we  implement Algorithm \ref{alg: 1} to obtain the approximation $u_N^h(t)$ (under $\pr_x$) of $u(t,x)$, and  verify the results found in Theorems \ref{thm: l2error} and \ref{CLT}. In particular, the focus is on the behavior of the $L^2$ error as $h\to0$, and as $N\to+\infty$, as well as the confidence intervals for $u_N^{h(N)}(t)$.
    The $L^2$ error is computed as a mean squared error (MSE):
    \begin{equation}
        \frac{1}{M}\sum_{i=1}^M \left(^iu_N^h(t)-u(t,x)\right)^2,
    \end{equation}
    i.e. we take $M$ independent realizations of $u_N^h(t)$ and calculate its mean.
    Theorem \ref{CLT} allows us to write the confidence intervals
    \begin{equation}\label{eq: CI}
        u_N^{h(N)}(t) \pm \frac{\sigma(t,x)}{\sqrt{N}}\varphi(\alpha/2),
    \end{equation}
    where $\varphi(\alpha)$ is $\alpha$-quantile of the standard normal distribution, and $\sigma(t,x)\leq\|f\|_\infty$. Note that for a desired tolerance error $\varepsilon$, $N$ has to satisfy $N> \varepsilon^{-2}\|f\|_ \infty^2 \varphi(\alpha/2)^2$. We will compare the confidence intervals with the robust theoretical standard error ($\sigma(t,x)\le 1$) with the sampled one.

    Figure~\ref{fig: Error} shows results for the 2-dimensional ball of unit radius, the initial condition $f(x)=(1-|x|^2/R^2)^3$, $x\in\R^2$, and the stable parameter $\alpha=1/2$. In Subfigure~\ref{fig: L2_h}, for a fixed time-space point $(T,x)=(0.5,0)$, it is shown on a log-log scale how the $L^2$ error (with $M=20$) exhibits a decrease as the time step $h$ approaches zero for fixed $N=10^6$, together with a linear log-log fit (slope is 0.94). In Subfigure~\ref{fig: L2_N}, for a fixed time-space point $(T,x)=(0.5,0)$, it is shown on a log-log scale how the $L^2$ error (with $M=20$) exhibits a decrease as the number of Monte Carlo samples $N$ approaches infinity for fixed $h=10^{-3}$, together with a fitted curve of the expected type $MSE(N)\approx a+b/N$. Note that the disparity between the sample and the fit for big $N$ is a bit illusive due to log-log scale. In Subfigure~\ref{fig: L2_h+N}, we take variable time step $h$ and then take $N=\lfloor 1/h\rfloor$ and demonstrate linear decay (linear log-log fit has slope 1.01). These findings align with those in Theorem~\ref{thm: l2error}. Subfigure~\ref{fig: CLT bounds} shows $u_N^{h(N)}$, where $h(N)=N^{-1-10^{-4}}$, as well as its confidence intervals for a tolerance $\varepsilon$. In particular, we plotted the $95$\%--confidence interval and took $\varepsilon=10^{-2}$. The number of Monte Carlo samples is therefore $N=38416$, which implies $h(N)\approx 2.6\times10^{-5}$. We emphasized the difference between the a priori theoretical bound for the variance  in the confidence interval (i.e., the blue area drawn using the a priori bound for $\sigma(t,x)$ which is $\|f\|_\infty=1$, and which lead to $N=38416$) and the confidence interval with the sampled variance (the orange area). This suggest that the true number of samples $N$, needed to reach the desired tolerance, could be far smaller than the theory necessitates.
    
\end{example}

\begin{figure}[ht]
	\centering
	\subcaptionbox{behavior of the mean squared error as the time step decreases.\label{fig: L2_h}}[.48\textwidth]{
		\includegraphics[width=\linewidth]{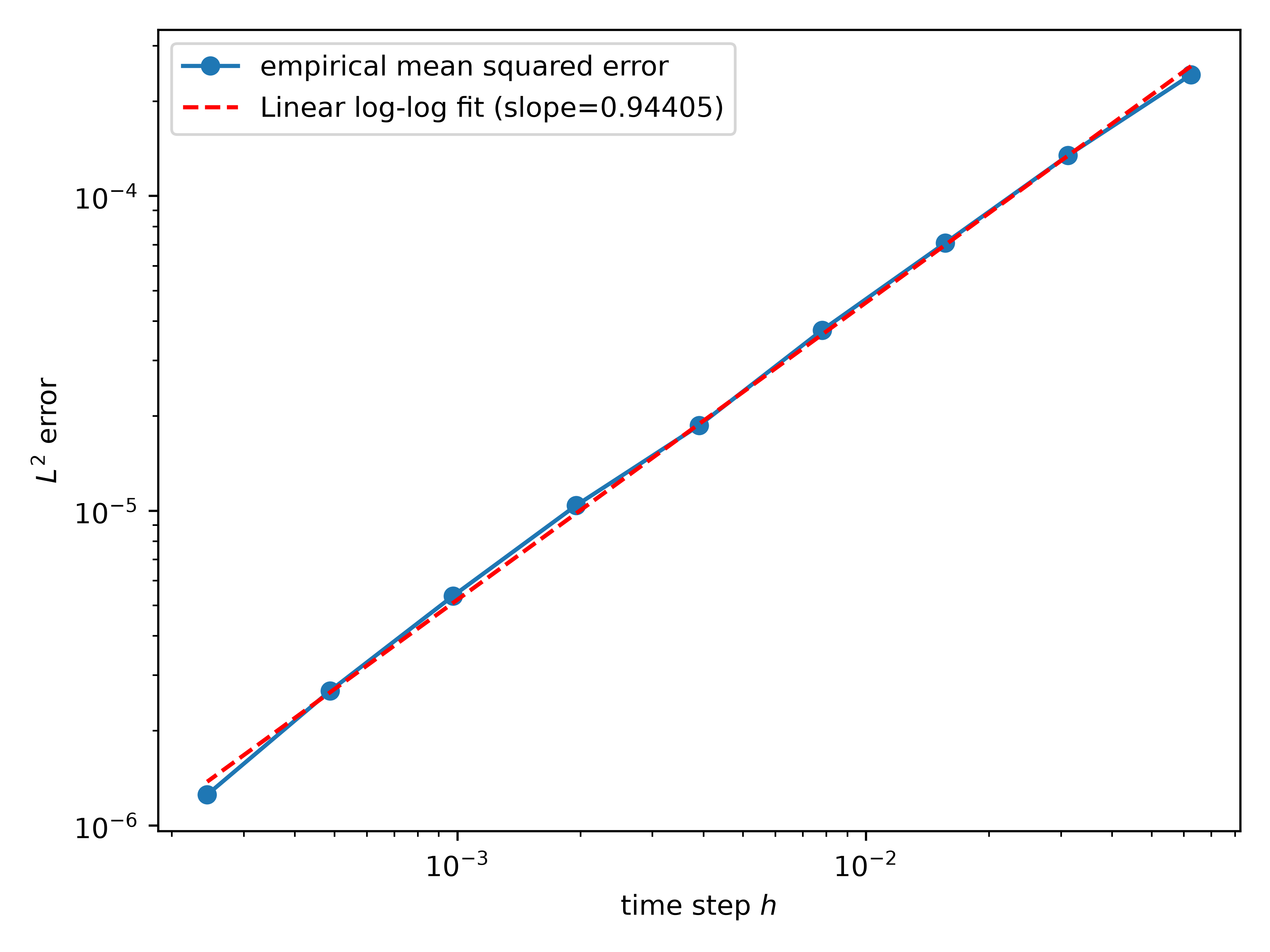}
	}
	\hfill
	\subcaptionbox{behavior of the mean squared error as the number of Monte Carlo samples increases.\label{fig: L2_N}}[.48\textwidth]{
		\includegraphics[width=\linewidth]{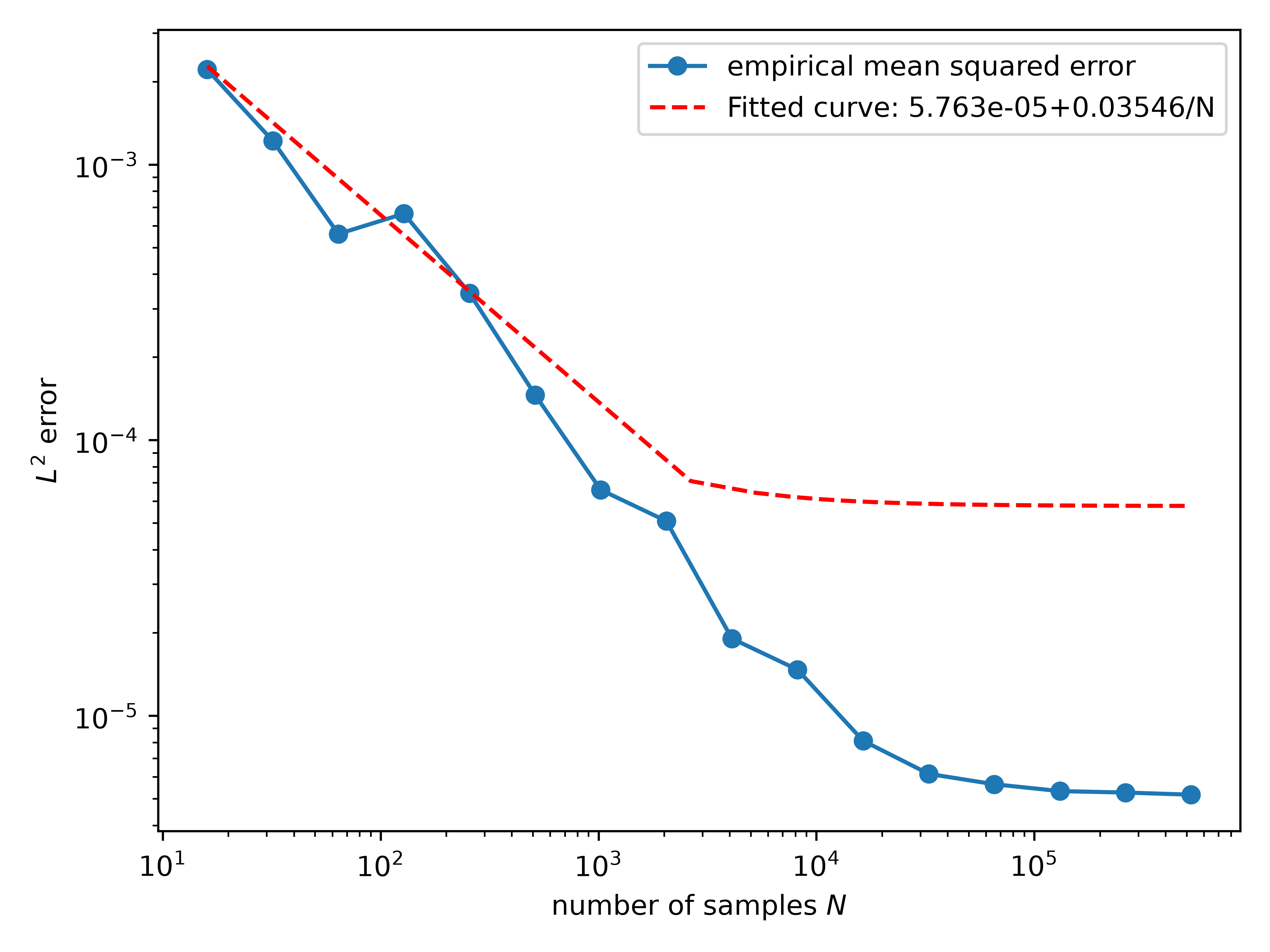}
	}
	\par\medskip
	\subcaptionbox{behavior of the mean squared error with joint decrease $h=1/N$.\label{fig: L2_h+N}}[.48\textwidth]{
		\includegraphics[width=\linewidth]{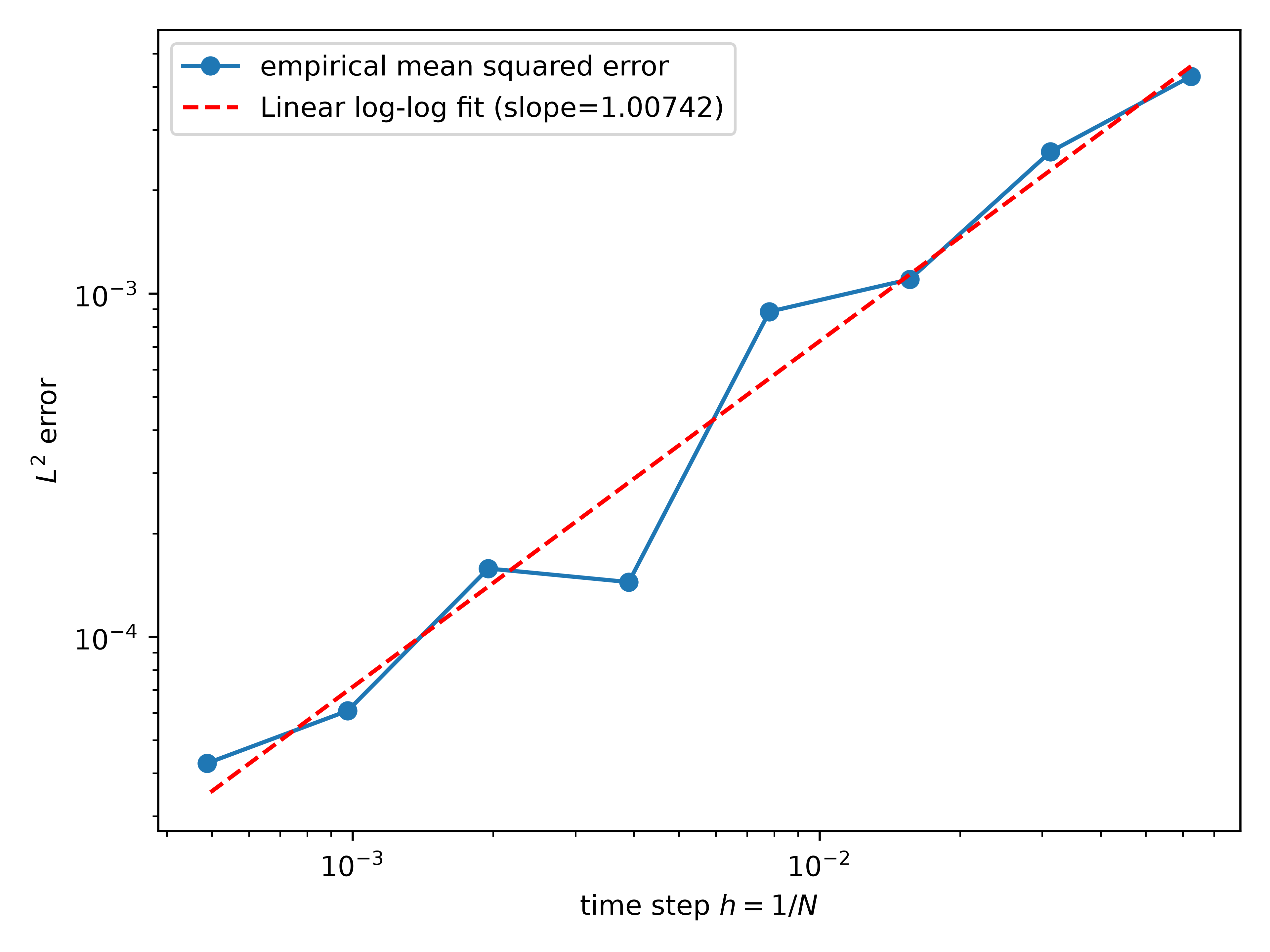}
	}
	\hfill
	\subcaptionbox{The Monte Carlo estimator $u_N^{h(N)}$ and confidence interval.\label{fig: CLT bounds}}[.48\textwidth]{
		\includegraphics[width=\linewidth]{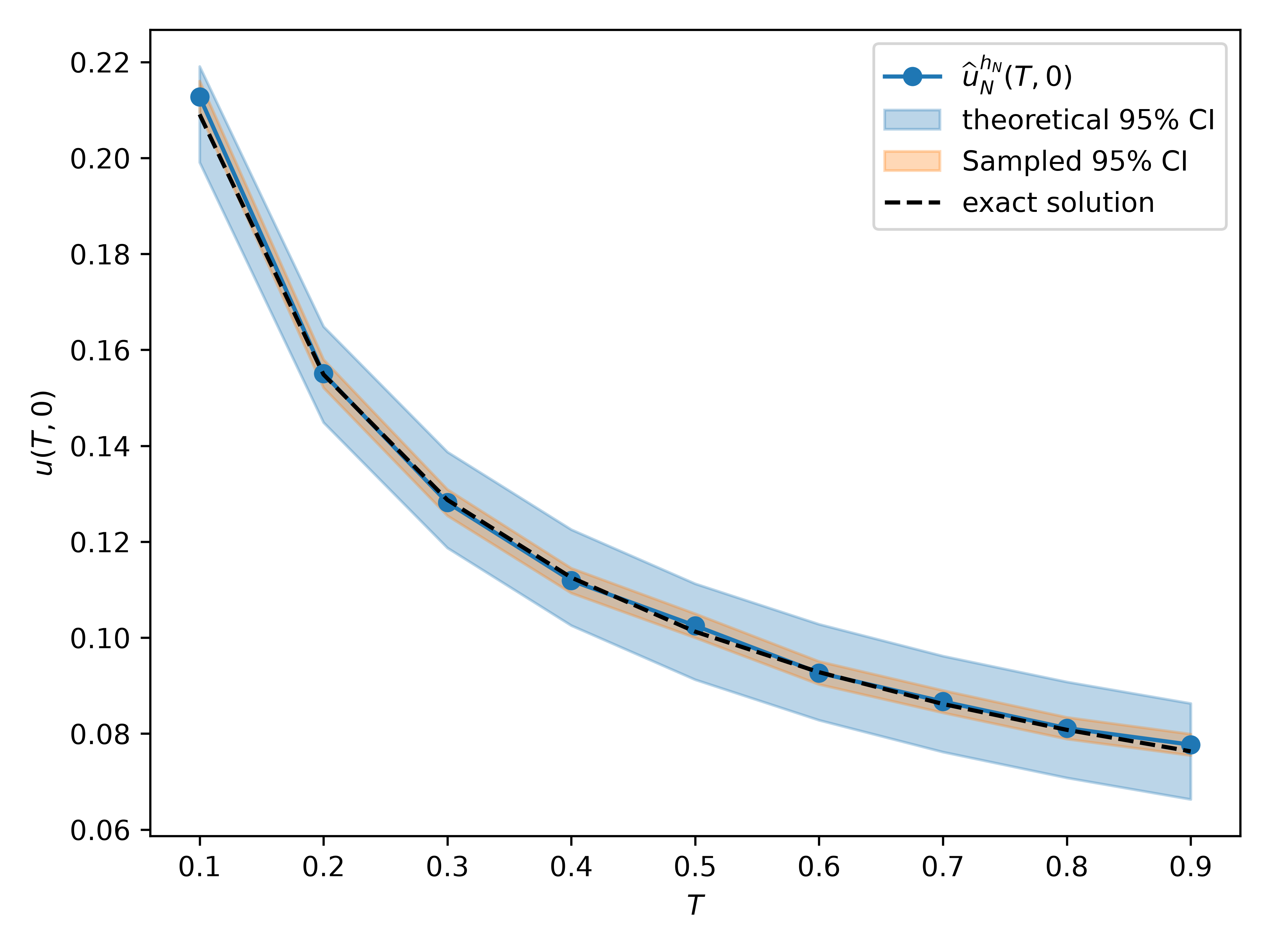}
	}
		\caption{For fixed $M=20$ and $(T,\,x)=(1/2,\,0)$, subfigure (A) shows the decrease of the MSE (mean squared error) as the time step $h$ approaches zero, with $N=10^^6$ fixed; while in subfigure (B) the MSE decreases as number of Monte Carlo samples $N$ approches infinity, with $h=10^{-3}$ fixed. Subfigure (C) shows the decrease of the MSE as the step $h$ approaches zero as number of samples increases jointly with $h$ by $N=\lfloor 1/h\rfloor$. Finally, for $x=0$ fixed, subfigure (C) shows the 95\% confidence intervals, shaded regions, for a desired tolerance of $\varepsilon=10^{-2}$; here $N=38417$ and $h=2.6\times10^{-5}$.}
		\label{fig: Error}
\end{figure}

\begin{example}\label{ex:high-dimensional-shell}
	The following example illustrates the applicability of our Monte Carlo
	method in a genuinely high-dimensional setting with non-trivial geometry,
	in which a deterministic spatial discretization would become prohibitively
	expensive.
	
	More precisely, at a fixed point $x_\star$, we approximate the solution
	\[
	u_d(T,x_\star)
	=
	\mathds E_{x_\star}\!\left[
	f_d(X_{L_T})
	\mathbf 1_{\{L_T<T_{D_{A,d}}\}}
	\right]
	\]
	to a time-fractional Cauchy--Dirichlet problem with initial datum $f_d$ on
	the $d$-dimensional anisotropic shell $D_{A,d}$.

    The shell geometry is relevant in applications. For example, in three
    spatial dimensions, hollow spherical shells arise naturally in models of
    diffusion-controlled release from porous nanocarriers. In such models, mass
    transport takes place through a shell surrounding an inner cavity, and the
    release kinetics depend on the shell thickness; see, e.g.,
    \cite{WangEdwards2016}. The high-dimensional anisotropic shell considered
    below is therefore a computational extension of a geometry that already
    appears in diffusion-driven transport models.
    
    We choose the initial datum $f_d$ to be a normalized Dirichlet eigenfunction
    of radial-like type. As we prove below, this implies that
    $u_d(T,x_\star)$ admits a closed form, which can be used as an exact benchmark
    for the Monte Carlo approximation. The radial-like structure of $f_d$, as
    well as the various parameters involved in the construction below, is used
    only to compute the reference solution exactly; namely, it reduces the
    high-dimensional problem to a one-dimensional radial spectral problem; see
    Appendix \ref{ap:b2}. Our sampling algorithm does not use this
    one-dimensional radial reduction: it evolves all $d$ coordinates of every
    path.
    
    We now define the setting rigorously. The technical details are provided in
    Appendix \ref{ap:b2}. Let $d\geq2$, put $\nu=d/2-1$, fix $\rho\in(0,1)$ and $\kappa>0$, and define
    \begin{equation*}
        A_d\coloneqq\operatorname{diag}(a_1,\ldots,a_d),
        \qquad
        a_i\coloneqq
        \begin{cases}
            1, & i \text{ odd},\\
            1/2, & i \text{ even}.
        \end{cases}
    \end{equation*}
    Consider the generator 
    \[\mathcal{G}g(x)=\sum_{i,j=1}^d(A_dA_d^T)_{ij}\partial^2_{x_ix_j}g(x),
    \]
    which is the infinitesimal generator of the process $X_t=x+\sqrt{2}\,A_dB_t$, $t\ge0$, where $(B_t)_t$ is the standard Brownian motion in $\R^d$. For the time
    operator, let $0<\alpha<1$ and consider the Caputo fractional derivative $\partial_t^\alpha f(t)=\frac{1}{\Gamma(1-\alpha)}\int_0^t f'(s)(t-s)^{-\alpha}ds$, which corresponds to the $\alpha$-stable inverse subordinator.
    
    For the domain of the time-fractional Cauchy--Dirichlet problem take
    \begin{equation*}
    	D_{A,d}\coloneqq
    	\left\{x\in\R^d:\ R_0<|A_d^{-1}x|<R_1\right\},
    \end{equation*}
    where $R_0$ and $R_1$ are specified below. Note that $D_{A,d}$ is a smooth anisotropic shell with two boundary components. In the definition of $D_{A,d}$, we take $R_1\coloneqq\frac{q_{\nu,1}}{\kappa}$ and $R_0\coloneqq\rho R_1$ where $q_{\nu,1}>0$ is the square root of the first eigenvalue of the problem \eqref{eq:radial-SL}.
    Under this setting, there exists a radial-like non-negative eigenfunction $f_d$ associated with the eigenvalue $-\kappa^2$ of the operator $\mathcal G$. More precisely, 
    \[
    	\mathcal{G} f_d=-\kappa^2 f_d\quad \textrm{in $D_{A,d}$},\qquad  f_d=0\quad \textrm{on $\partial D_{A,d}$},
    \]
    where $z\mapsto f_d(A_dz)$ is a true radial function in the spherical shell $\{z\in\R^d:R_0<|z|<R_1\}$. Moreover, since $\mathcal{G} f_d=-\kappa^2 f_d$ in $D_{A,d}$, $\mathcal{G}f_d$ continuously extends to 0 at $\partial D_{A,d}$, i.e. $f_d$ satisfies \ref{assumption_f_}. We normalize $f_d$, and with $x_\star\in D_{A,d}$ we denote the point of maximum of $f_d$, i.e. $f_d(x_\star)=\|f_d\|_\infty=1$. Note that $x_\star$ is not unique since  $f_d$ is radial-like so we choose it in the positive direction of the first coordinate vector $e_1$. For details, see \eqref{def:f_d}. Moreover, by the representation \eqref{def:f_d}, the function $f_d$ can be easily evaluated in all points in $D_{A,d}$.
    
    Let $(P_s^D)_{s}$ denote the semigroup of the underlying diffusion $X_t=x+\sqrt 2 A_dB_t$ 
    killed upon exiting $D_{A,d}$. The eigenfunction relation implies $P_s^D f_d(x)  =    e^{-\kappa^2s}f_d(x)$,  $s\geq0$, $x\in D_{A,d}$.
    Therefore, conditioning on the inverse stable clock $L_T$ and using its
    independence from $X$, we obtain
    \[
    	u_d(T,x)= \mathds E_x\left[f_d(X_{L_T})\mathbf 1_{\{L_T<T_{D_{A,d}}\}} \right]=\mathds E\left[P_{L_T}^D f_d(x)\right] =
    	f_d(x)\,\mathds E\left[ 	e^{-\kappa^2 L_T} 	\right] = f_d(x)E_\alpha(-\kappa^2 T^\alpha),
    \]
    where $E_\alpha$ denotes the Mittag--Leffler function.  For more details, see e.g. \cite[Theorem~3.1, Eq.~(3.5)]{MeerschaertNaneVellaisamy2009}.   Finally, since $f_d(x_\star)=1$, we have
    \begin{equation}\label{eq:highdim-exact-solution}
    	u_d(T,x_\star)
    	=
    	E_\alpha(-\kappa^2T^\alpha),
    \end{equation} 
   which is the function that we approximate by our simulations. We emphasize once more that this elaborate construction is introduced solely to obtain an exact solution of the time-fractional Cauchy–Dirichlet problem, which serves as a benchmark for our simulation-based approximation.
    
    For the numerical experiment, fix now $\rho=4/5$, $\kappa=2$ and $\alpha=1/2$. In the main part of the numerical example, we will also take the dimension $d=20$ in order to highlight the difference between a pointwise Monte Carlo calculation and a global space--time discretization. Indeed, a conventional tensor-product grid on a bounding box with only eight interior points in each coordinate direction already contains
    \begin{equation*}
        8^{20}=1{\,}152{\,}921{\,}504{\,}606{\,}846{\,}976
    \end{equation*}
   points, even before discretizing the non-local time history. Assumptions under which sparse grids reduce the computational complexity may be found, e.g., in \cite[pp.~148--149]{BungartzGriebel2004}. In contrast, Algorithm~\ref{alg: 1} computes the single value $u_d(T,x_\star)$ without constructing a spatial mesh: at each Euler time point, it only checks
    \(
        R_0<|A_d^{-1}Y_{t_i}|<R_1,
    \)
    and stores $d$ coordinates for each active path. Moreover, the independent
    paths can be simulated in parallel.

    To numerically compute the values $u_d(T,x_\star)$, note that \(E_{1/2}(-z)
        =\operatorname{erfcx}(z),
    \)
    where $\operatorname{erfcx}$ is usually called \textit{scaled complementary error function} and is implemented in major numerical libraries. Hence,
    \begin{equation*}
        u_{20}(T,x_\star)=\operatorname{erfcx}(4\sqrt{T}).
    \end{equation*}
    For sampling, we also require the following values
    \begin{align*}
        q_{9,1}&\approx18.623363476094049,\\
        R_0&\approx7.449345390437620,
        \qquad R_1\approx9.311681738047024,\\
        x_\star&\approx(8.025577759448735,0,\ldots,0).
    \end{align*}
   
   In order to determine the confidence interval, we note that it is also possible to compute the theoretical variance entering \eqref{eq: CI}
   independently of the Monte Carlo samples, with great precision using eigenfunctions expansion of the semigroup $P_s^D$. For details, see \eqref{variance}.

    Figure~\ref{fig:high-dimensional-shell} reports the numerical results. Subfigure~\ref{fig:highdim-geometry} shows the $(x_1,x_2)$ section of the twenty-dimensional anisotropic shell. 
    
    In Subfigure~\ref{fig:highdim-scaling}, we fix the values $N=20000$, $T=10^{-2}$, and $h=10^{-3}$, while the dimension $d$ varies from $2$ to $32$, and we measure the time needed to obtain the approximation of $u_{20}(T,x_\star)$. The drawn points are medians of 10 single-threaded runs, where the fitted log--log slope is $1.014$, which is consistent with linear work per Brownian step in the dimension.
    
    In Subfigure~\ref{fig:highdim-mse}, we report how the empirical mean squared error (MSE) behaves relative to the (practically) exact MSE $\sigma(T,x_\star)^2/N$. The run uses $d=20$, $M=50$ independent estimators, and the sequence
    \[
        h_N=N^{-21/20},\qquad 
       N\in\{2^6,2^7,2^8,2^9,2^{10}\}.
    \]
    The fitted MSE slope is $-1.08$, and the dashed line is the theoretical Monte Carlo contribution $\sigma_d(T,x_\star)^2/N$.
    
    Finally, Subfigure~\ref{fig:highdim-ci} uses $N=512$ and
    \(
        h_N=\frac{1}{5}N^{-21/20}\approx2.86\times10^{-4}.
   \)
    The shaded pointwise $95\%$ intervals are
    \begin{equation*}
        u_N^{h_N}(T,x_\star)
        \pm z_{0.975}\frac{\sigma_d(T,x_\star)}{\sqrt{N}},
    \end{equation*}
    with $\sigma_d(T,x_\star)$ is (practically) exact standard deviation; i.e., no empirical variance is used in their construction.
\end{example}

\begin{figure}[ht]
    \centering
    \subcaptionbox{The $(x_1,x_2)$ section of the anisotropic shell $D_{A,d}$ and the starting point $x_\star$.\label{fig:highdim-geometry}}[.48\textwidth]{
        \includegraphics[width=\linewidth]{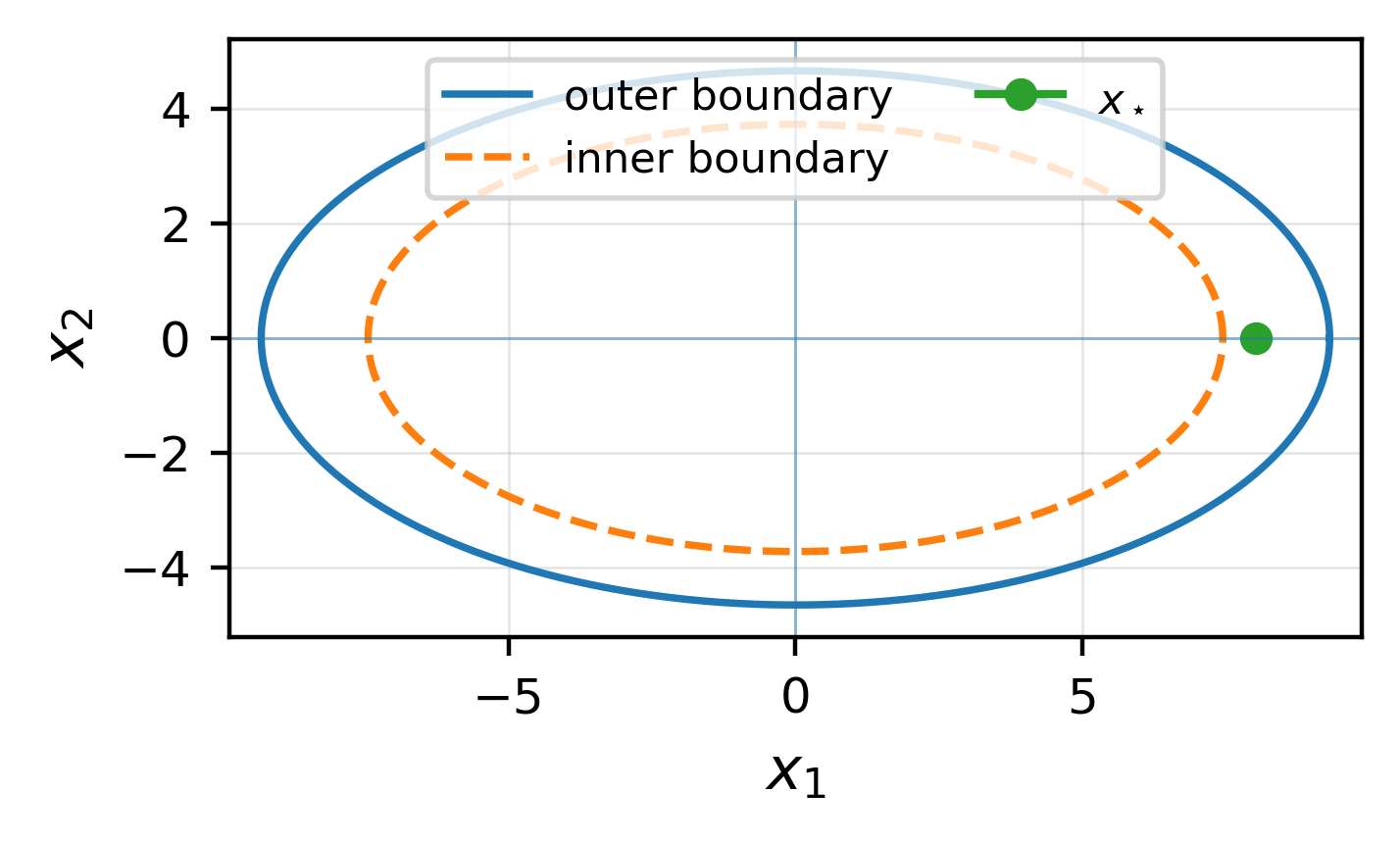}
    }
    \hfill
    \subcaptionbox{Single-core execution time as the dimension increases.\label{fig:highdim-scaling}}[.48\textwidth]{
        \includegraphics[width=\linewidth]{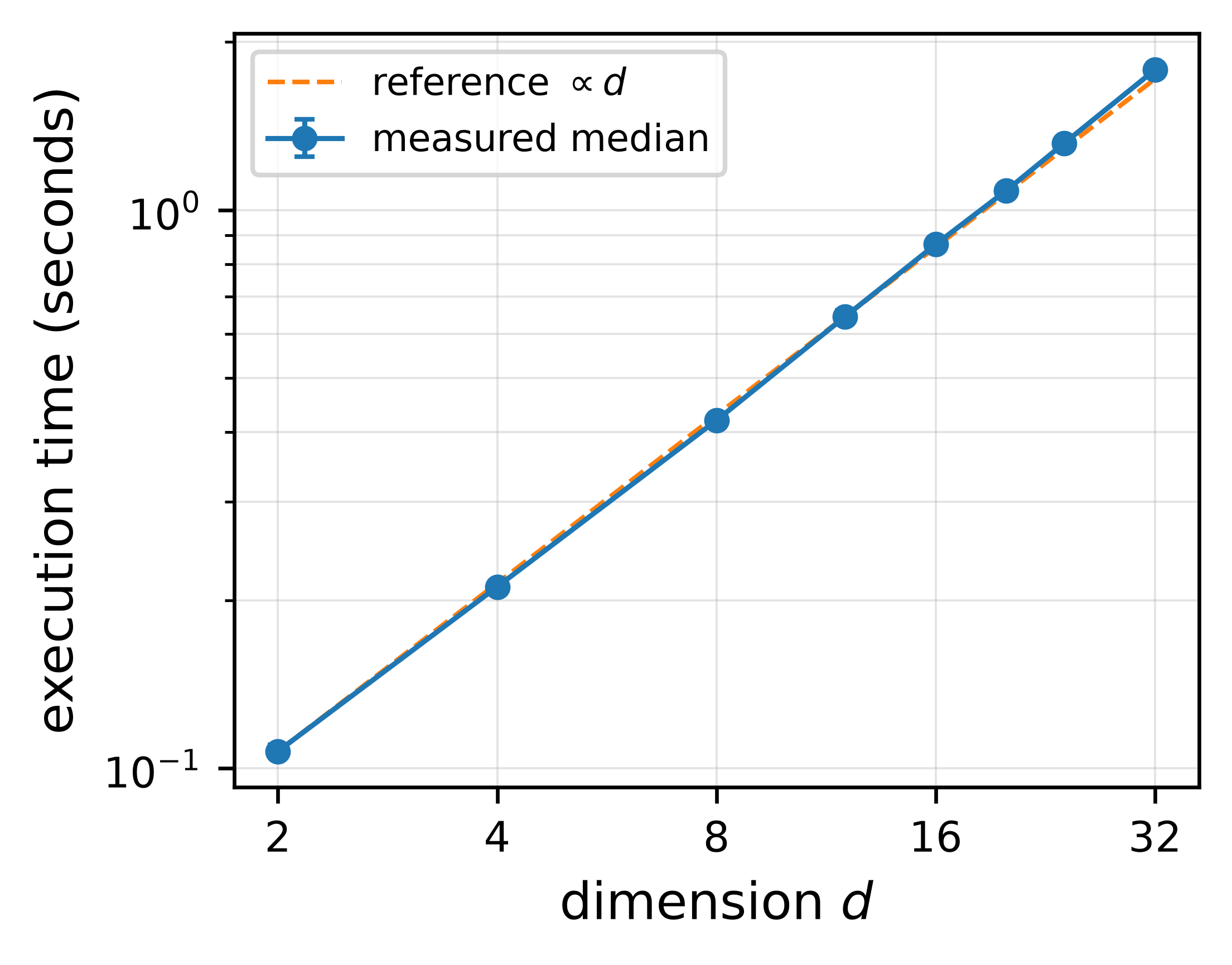}
    }
    \par\medskip
    \subcaptionbox{Empirical MSE in dimension $d=20$.\label{fig:highdim-mse}}[.48\textwidth]{
        \includegraphics[width=\linewidth]{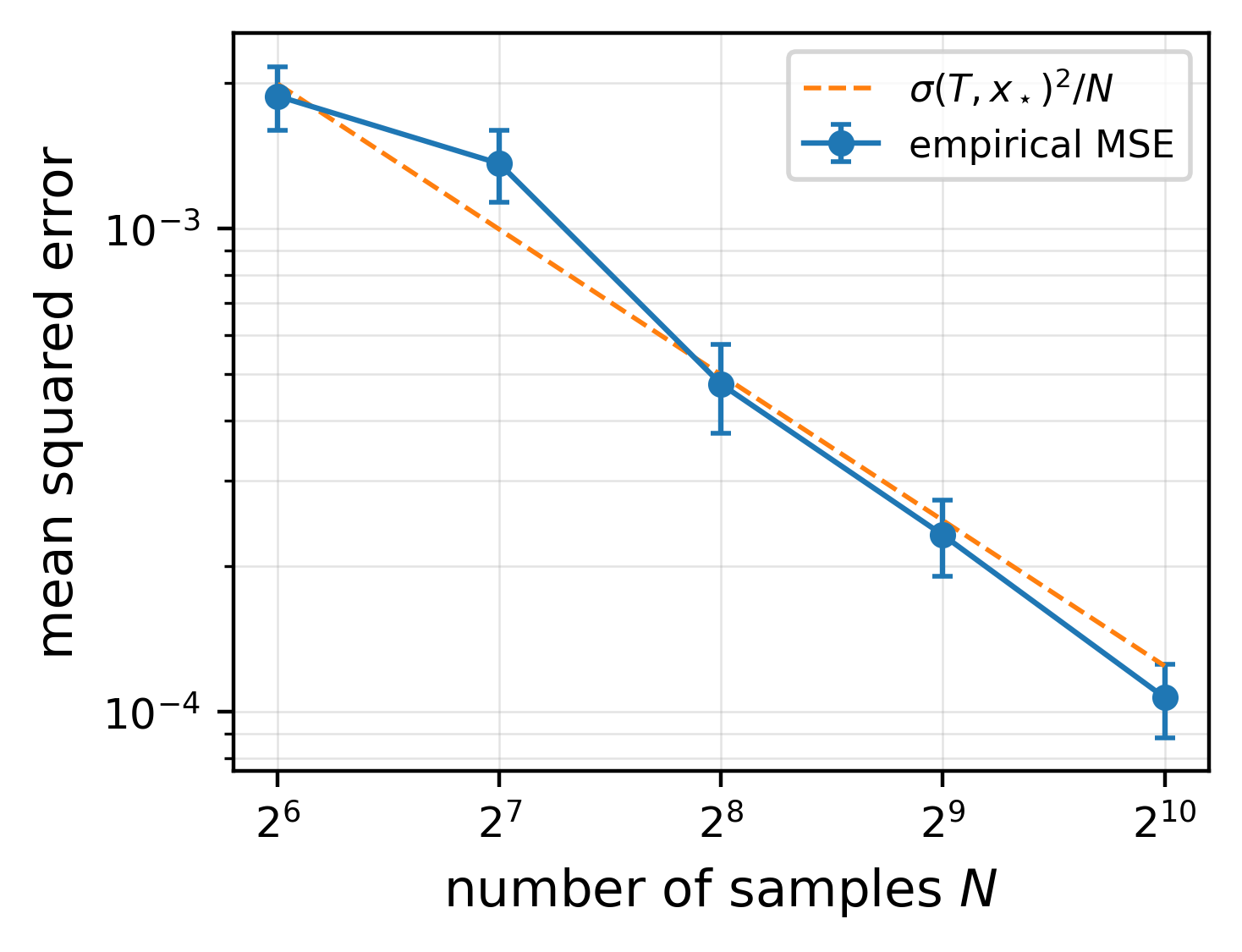}
    }
    \hfill
    \subcaptionbox{Estimator and pointwise $95\%$ intervals using the theoretical variance.\label{fig:highdim-ci}}[.48\textwidth]{
        \includegraphics[width=\linewidth]{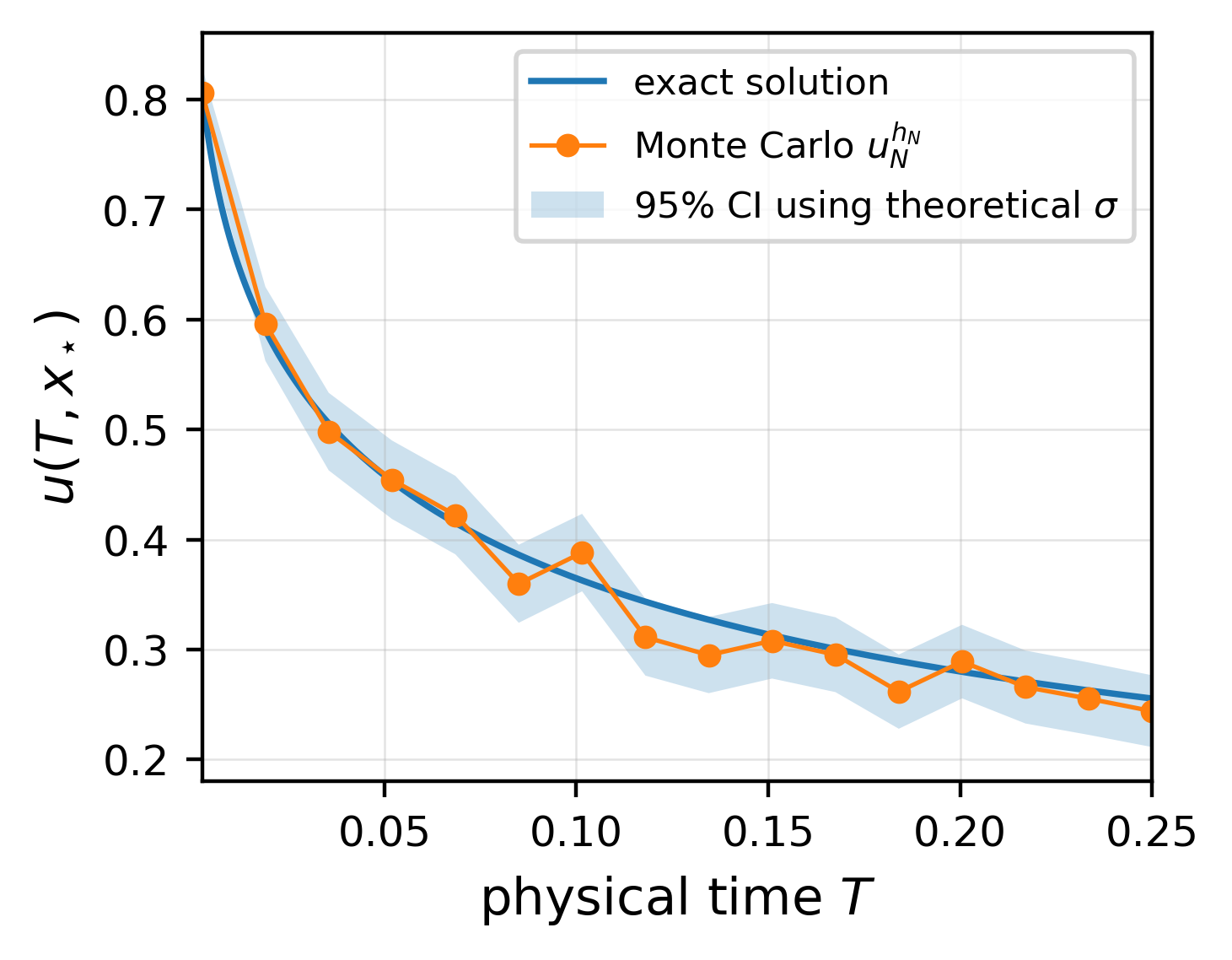}
    }
    \caption{In (A) the high-dimensional ellipsoidal-shell benchmark with $\rho=4/5$, $\kappa=2$, $\alpha=1/2$, and the alternating diagonal matrix $A_d$. In (B), $N=20000$, $T=10^{-2}$, $h=10^{-3}$, and the vertical bars show the interquartile range over 10 runs; the dashed line is proportional to $d$, with log-log slope 1.014. In (C), $d=20$, $T=10^{-2}$, $M=50$, and $h_N=N^{-21/20}$; the vertical bars are one standard error of the empirical MSE and the dashed line is $\sigma_d(T,x_\star)^2/N$. In (D), $d=20$, $N=512$, and $h_N=N^{-21/20}/5$; the solid curve is \eqref{eq:highdim-exact-solution}, the markers are the Monte Carlo estimates, and the shaded intervals use the theoretical variance \eqref{eq:highdim-theoretical-variance}.}
    \label{fig:high-dimensional-shell}
\end{figure}

	\section*{Acknowledgements}
	The authors acknowledge financial support under the National Recovery and Resilience Plan (NRRP), Mission 4, Component 2, Investment 1.1, Call for tender No. 104 published on 2.2.2022 by the Italian Ministry of University and Research (MUR), funded by the European Union – NextGenerationEU– Project Title “Non–Markovian Dynamics and Non-local Equations” – 202277N5H9 - CUP: D53D23005670006 - Grant Assignment Decree No. 973 adopted on June 30, 2023, by the Italian Ministry of University and Research (MUR)

    IB acknowledges financial support by the European Union – NextGenerationEU through the National Recovery and Resilience Plan 2021-2026 Institutional grant of University of Zagreb Faculty of Science (IK IA 1.1.3. Impact4Math), as well as the support by Croatian Science Foundation through the project IP-2025-02-8793.

    AM was supported in part by EPSRC grants EP/V009478/1 and EP/W006227/1.

    The authors would like to thank the Isaac Newton Institute for Mathematical Sciences, Cambridge, for support and hospitality during the programme Stochastic systems for anomalous diffusion, where work on this paper was undertaken. This work was supported by EPSRC grant EP/Z000580/1.
	
	\appendix	

\section{Technical results}
    First two lemmas are concerned with the density of the (killed) Brownian motion, given in \eqref{transition} and \eqref{hunt}.
	\begin{lemma}\label{ap:l:dens-reg}
		Assume \ref{assumption_domain}. For all $x,y\in D$, $t>0$, and a multi-index $\beta$, $|\beta|\le 2$, it holds that
		\begin{align}\label{1639}
			|\partial^\beta_xp_D(t,x,y)|\le \frac{C_0(d,D,\Sigma,\mu)}{t^{(d+|\beta|)/2}}e^{-C_1(\Sigma)\frac{|x-y|^2}{t}},\quad x,y\in D, t>0.
		\end{align}
	\end{lemma}
    \begin{proof}
It is enough to prove the claim  with zero drift. Indeed, if $p_D^0$ denotes the killed density corresponding to the same
covariance matrix $\Sigma$, but with drift $\mu=0$, then
by Cameron--Martin--Girsanov formula,
for every Borel set $A\subset D$,
\[
        \mathbb P_x^\mu(X_t\in A,\ t<T_D)
        =
        \mathds E_x^0\left[
        \exp\left(\Sigma^{-1}\mu\cdot(X_t-x)-\frac{\mu^T\Sigma^{-1}\mu t}{2}\right)
        \mathbf 1_{\{X_t\in A,\ t<T_D\}}
        \right].
\]
Since the exponential factor depends only on $X_t$, this identity implies the
kernel relation
\begin{equation}\label{eq:killed-conjugation}
        p_D(t,x,y)
        =
        \exp\left(\Sigma^{-1}\mu\cdot(y-x)-\frac{\mu^T\Sigma^{-1}\mu t}{2}\right)p_D^0(t,x,y).
\end{equation}
Since $D$ is bounded, differentiating the exponential factor in
\eqref{eq:killed-conjugation} shows that it is enough to prove
\eqref{1639} for $\mu=0$.  In the sequel we therefore assume $\mu=0$
and write $p_D$ instead of $p_D^0$.

It is well known and easy to see by a direct computation that $p_D(t,x,y)$  is locally smooth both in $x$ and in $t$ and that it pointwise solves the parabolic problem 
        \begin{align}\label{eq:homogeneous-parabolic}
            \partial_t p_D(t,x,y)=\frac12\sum_{i,j=1}^d{\Sigma_{ij}}\partial^2_{x_i x_j}p_D(t,x,y),
        \end{align}
        in $(t,x)\in (0,+\infty)\times D$ for all $y\in D$. Moreover, since $D$ is of class {$C^{3+\alpha}$}, {it is known that the kernel $t\mapsto p_D(t,x,y)$ is of class $C^\infty$ and $x\mapsto p_D(t,x,y)$ is of class $C^{3+\alpha}(\overline D)$ (and even $\partial_t p_D(t,x,y)$ as well)}. {This (joint) regularity follows from the spectral representation of $p_D(t,x,y)$
        \begin{align}
            p_D(t,x,y)=\sum_{j=1}^\infty e^{-\lambda_jt}\varphi_j(x)\varphi_j(y),
        \end{align}
        where $(\lambda_j,\varphi_j)$, $j\in\N$, are the eigenpairs of $-\frac12\sum_{i,j=1}^d{\Sigma_{ij}}\partial^2_{x_i x_j}$, together with Weyl's law $\lambda_j\sim j^{2/d}$ and elliptic Schauder regularity up to the boundary for the Dirichlet eigenfunctions $\varphi_j$; see \cite[Chapters 6 \& 8]{Gilbarg1977}, and \cite[Section 2.6]{FernandezR-RosO-PDEs}. For the detailed proof for the Laplacian in a $C^{1,1}$ domain see \cite[Lemma A.7]{Bio23}.} This regularity means that we are free to use Schauder's regularity theorems both in the interior and on the boundary from \cite[Chapter 3 \& Chapter 4]{Friedman1964}.

\medskip
\noindent\textit{Step 1: the estimate for $0<t\leq 1$.}

Let $0< t\le 1$, fix $y\in D$. We first deal with the interior estimate. Let $r_0>0$ denote a constant dependent only on $D$ (which final value will be clear after the boundary estimate) such that $r_0\le R/8$ where we recall that $R$ is the localization radius of Property \ref{property 3.1}.
If $\dist(x,\partial D)\geq r_0\sqrt t$, then the cylinder
\[
        Q:=\left(\frac{t}{2},\frac{3t}{2}\right)
        \times B\left(x,\frac{r_0\sqrt t}{2}\right)
\]
is contained in $(0,\infty)\times D$.  Applying the weighted interior
Schauder estimate \cite[Chapter~3, Section~2, Theorem~5]{Friedman1964}
to the homogeneous equation~\eqref{eq:homogeneous-parabolic} gives, for
$|\beta|\le2$,
\begin{equation}\label{eq:interior-local-estimate}
        |\partial_x^\beta p_D(t,x,y)|
        \leq
        C(\Sigma) t^{-|\beta|/2}
        \sup_{(\tau,\xi)\in Q} p_D(\tau,\xi,y).
\end{equation}
Here, the factor $t^{-|\beta|/2}$ comes from the parabolic distance from
$(t,x)$ to the parabolic boundary of $Q$, which is comparable to
$\sqrt t$, see \cite[Chapter 3, Section 2]{Friedman1964}. Further, note that
\begin{align}
	 \sup_{(\tau,\xi)\in Q} p_D(\tau,\xi,y)\le  \sup_{(\tau,\xi)\in Q} p(\tau,\xi,y)\le \sup_{\xi\in B(x,r_0\sqrt{t}/2)}\frac{c_0}{t^{d/2}}e^{-c_1(\Sigma)\frac{|\xi-y|^2}{t}}.
\end{align}
However, for $\xi \in B(x,r_0\sqrt{t}/2)$ it holds $|\xi-y|\ge |x-y|-r_0\sqrt{t}/2$, therefore
\begin{align}
	\sup_{\xi\in B(x,r_0\sqrt{t}/2)}\frac{c_0}{t^{d/2}}e^{-c_1(\Sigma)\frac{|\xi-y|^2}{t}}\le\frac{c_2(r_0)}{t^{d/2}}e^{-c_3(\Sigma)\frac{|x-y|^2}{t}}.
\end{align}
This gives the claim in the case $\dist(x,\partial D)\geq r_0\sqrt t$.

It remains to treat the case \(\dist(x,\partial D)<r_0\sqrt t\).  The strategy is to flatten the boundary, and use Schauder's boundary regularity theorem on the flattened boundary and modified heat equation to obtain the desired constant of the lemma. 

First we flatten the boundary by choosing finite number of
the charts $F^i$'s and $G^i$'s as in Property \ref{property 3.1}, together with the sets $U^i$'s and $V^i$'s. Let $L_\star\coloneqq \max_i\{\textrm{Lip}(F^i),\textrm{Lip}(G^i)\}$, where $\textrm{Lip}(\cdot)$ denotes the Lipschitz constants of a function. Note that we can make $r_0$ so small so that $3(L_\star+1)r_0\le \min\{1/2,R\}$ and so that for all $\xi\in D$ such that $\dist(\xi,\partial D)\le 3(L_\star+1)r_0$ it holds that $\sup_i\dist(\xi,(U^i)^c\cap D)\ge 3(L_\star+1)r_0$, i.e. in this area $\xi$ is always far away from the boundary of at least one $U^i$ in which is also contained in. By such choice of $r_0$, it is enough to consider the case where $x$ belongs to a generic $U^i$ for which $B(x,3(L_\star+1)r_0)\cap D\subset U^i$. 

Now define
\[
        w(t,z):=p_D(t,G^i(z),y).
\]
By the chain rule, \(w\) solves an equation of the form
\begin{equation}\label{eq:flattened-operator}
        \partial_t w(t,z)
        =
        \sum_{k,l=1}^d a^i_{kl}(z)\partial^2_{z_kz_l}w(t,z)
        +
        \sum_{k=1}^d b^i_k(z)\partial_{z_k}w(t,z),
\end{equation}
in the local half-cylinder $(0,\infty)\times \Big((0,R)\times V^i\Big)$, where
\begin{equation}\label{eq:flattened-a}
        a^i_{kl}(z)
        =
        \frac12\sum_{m,n=1}^d{\Sigma_{mn}}
        \partial_{x_m}F^i_k(G^i(z))
        \partial_{x_n}F^i_l(G^i(z)),
\end{equation}
\begin{equation}\label{eq:flattened-b}
        b^i_k(z)
        =
        \frac12\sum_{m,n=1}^d{\Sigma_{mn}}
        \partial^2_{x_mx_n}F^i_k(G^i(z)).
\end{equation}
The coefficients in \eqref{eq:flattened-operator} are now variable, but still uniformly elliptic and
have uniformly bounded H\"older norms, because \(\Sigma\) is positive definite,
since \(D\) is \({C^{3+\alpha}}\) {and so are $F^i$'s and $G^i$'s}, and only finitely many charts $F^i$'s and $G^i$'s are used.

Let $x_0\in \partial D$ so that $|x-x_0|=\dist(x,\partial D)$. Without loss of generality, by translation, we can also assume that $F^i(x_0)=0\in \R^d$. The
Dirichlet condition gives \(w=0\) on the flat boundary \(\{z_1=0\}\). After the
parabolic scaling 
\[
        \tau=t+r^2\theta,
        \qquad
        z=F^i(x_0)+r\zeta=r\zeta,
        \qquad \textrm{with }
        r:=3(L_\star+1)r_0\sqrt t,
\]
i.e. by putting $\wt w(\theta,\zeta)=w(t+r^2\theta,r\zeta)$ and modifying the equation \eqref{eq:flattened-operator} accordingly,
the scaled
coefficients of the  modified \eqref{eq:flattened-operator} still satisfy the hypotheses of \cite[Chapter~4,
Section~7, Theorem~4]{Friedman1964}, with constants depending only on
\(d,D,\Sigma\). Note that in the new rescaled coordinates, the point $(t,x)$ is actually $\left(0,\frac{F^i(x)}{3(L_\star+1)r_0\sqrt t}\right)$, where $\left|\frac{F^i(x)}{3(L_\star+1)r_0\sqrt t}\right|\le 1/3$. 

Since the boundary datum on the flattened boundary is zero and the forcing term in the (modified) equation \eqref{eq:flattened-operator} is zero, we now apply
\cite[Chapter~4,
Section~7, Theorem~4]{Friedman1964} to the domain $\DD=(-1,1)\times \Big((0,1)\times(-1,1)^{d-1}\Big)$, and boundary parts $R=(-1,1)\times \Big(\{0\}\cap(-3/4,3/4)^{d-1}\Big)$, and $R_0=(-1/2,1/2)\times \Big(\{0\}\cap(-1/2,1/2)^{d-1}\Big)$. We obtain for $(\theta,\zeta)=\left(0,\frac{F^i(x)}{3(L_\star+1)r_0\sqrt t}\right)$, which is of constant parabolic distance to  $\partial_p\DD\setminus R_0$, and for every multi-index $|\gamma|\le2$
\begin{align}\label{est-1}
    |\partial_\zeta^\gamma\wt w(\theta,\zeta)|\le C\sup_{(\theta,\zeta)\in\DD}|\wt w(\theta,\zeta)|,
\end{align}
where the constant $C=C(\DD,R,R_0,d,\Sigma,D)>0$ is the one of \cite[Chapter~4,
Section~7, Theorem~4]{Friedman1964}, and is by this construction independent of $x$ and $t$.

After rescaling and returning to the original coordinates, since $\wt w(\theta,\zeta)=u(t+r^2\theta,G^i(r\zeta))$ and since the first and second derivatives of \(G^i\) are bounded, we obtain
\begin{equation}\label{eq:boundary-local-estimate-x}
        |\partial_x^\beta p_D(t,x,y)|
        \leq
        C t^{-|\beta|/2}
        \sup_{(\tau,\xi)\in \widetilde Q}p_D(\tau,\xi,y),
        \qquad |\beta|\leq2,
\end{equation}
where \(\widetilde Q\subset(0,\infty)\times D\) is contained in
\([t/2,3t/2]\times\Big(B(x,C_D\sqrt{t})\cap D\Big)\), for some $C_D>0$.
Now, in a similar way as for \eqref{eq:interior-local-estimate}, we get 
\begin{align}
    \sup_{(\tau,\xi)\in \wt Q}p_D(\tau,\xi,y)\le \frac{c_5(d,D,\Sigma)}{t^{d/2}}e^{-c_6(\Sigma)\frac{|x-y|^2}{t}}.
\end{align}
This finishes the proof for $0<t\le 1$.

\noindent\textit{Step 2: $t\ge 1$.}
Let first $1\le t\le2$. The semigroup property and the estimate already proved at
time $1/2$ give
\[
        |\partial_x^\beta p_D(t,x,y)|
        \leq
        \int_D |\partial_x^\beta p_D(1/2,x,z)|p_D(t-1/2,z,y)d z
        \leq C.
\]
The right-hand side of \eqref{1639} is also bounded from
below by a positive constant on $1\le t\le2$,  since $D$ is bounded. This finishes the proof for $1\le t\le 2$.

For $t\ge2$, the semigroup property  gives
\begin{align*}
        \partial_x^\beta p_D(t,x,y)
        & =
        \int_D\int_D
        \partial_x^\beta p_D(1/2,x,z)
        p_D(t-1,z,w)
        p_D(1/2,w,y)
        dz dw .
\end{align*}
Using Cauchy-Schwarz's inequality and the principal $L^2$--eigenvalue property for the semigroup of $X^D$, see \cite[Chapter~1, Section~1.3]{Davies1989}, we get
\begin{align*}
        |\partial_x^\beta p_D(t,x,y)|
        &\le
        \|\partial_x^\beta p_D(1/2,x,\cdot)\|_{L^2(D)}
        \|P_{t-1}^D p_D(1/2,\cdot,y)\|_{L^2(D)}        \\
        &\le
        C e^{-\lambda_D(t-1)},
\end{align*}
where $\lambda_D$ denotes the principal  eigenvalue of $-\frac12\sum_{i,j=1}^d{\Sigma_{ij}}\partial^2_{x_i x_j}$.
The two $L^2$-norms are uniformly bounded in $x,y\in D$ by the estimate
proved in the first step at time $1/2$.  Finally, for $t\ge2$,
$e^{-\lambda_D(t-1)}\le C_m t^{-m}$ for any $m>0$, and, since
$|x-y|\le \diam(D)$ we also have $
        \exp\left(-C_1\frac{|x-y|^2}{t}\right)
        \ge
        \exp\left(-C_1\diam(D)^2/2\right).
$
This finishes the proof for $t\ge 2$.

\end{proof}

\begin{lemma}\label{lemma bounds transition density}
    Let $P_t^\mu(x,y)$ be the density function as in \eqref{transition}, and let $D$ be an arbitrary bounded domain. Then there exists a positive constant $C=C(d,D,\Sigma,\mu)$ such that  for all $x,\,y\in D$ and all $t>0$ we have 
    \begin{align*}
        C^{-1} \leq \frac{P_t^\mu(x,y)}{P_t^0(x,y)} e^{\frac{\mu^T\Sigma^{-1}\mu}{2}\, t} \leq C.
    \end{align*}
\end{lemma}
\begin{proof}
    The ratio between $P_t^\mu(x,y)$ and $P_t^0(x,y)$ is
    \begin{align*}
        {\frac{P_t^\mu(x,y)}{P_t^0(x,y)} = \exp\left( -\frac{\mu^T\Sigma^{-1}\mu}{2}\, t + \mu^T \Sigma^{-1}(y-x)\right).}
    \end{align*}
    The result now follows since $D$ is bounded.
\end{proof}

\begin{lemma}\label{lemma 3.1}
        {Assume that \ref{assumption_domain} and \ref{assumption_f} are satisfied and that $v$ is as in \eqref{feynmann-kac}.} Then, for all $(s,x)\in[0,T)\times V_{\partial D}(\varepsilon)$ and every multi-index $\alpha$, $|\alpha|\le 2$ there exists a constant $C(|\alpha|,d,D,\Sigma,\mu)>0$ such that
    \begin{align*}
        |\partial_x^\alpha v(s,x)|\leq C\frac{\|f\|_\infty}{\varepsilon^{|\alpha|}}.
    \end{align*}
    If instead \ref{assumption_domain} and \ref{assumption_f_} are satisfied, then there exists $C=C(d,D,\Sigma,\mu,\alpha)>0$ such that
    \begin{align*}
        \sup_{0\le s<T}\sup_{x\in\overbar D}
        \left(|\nabla v(s,x)|+|D^2v(s,x)|\right)
        \le C\|f\|_{2+\alpha,\overbar D}.
    \end{align*}
    
\end{lemma}
\begin{proof}
    Assume \ref{assumption_domain} and \ref{assumption_f}. By using Lemma \ref{ap:l:dens-reg} and constants $C_0$ and $C_1$ therein, we have
    \begin{align*}
        |\partial_x^\alpha v(s,x)| &\leq \frac{C_0 \|f\|_\infty}{(T-s)^{(d+|\alpha|)/2}} \int_{\supp(f)} \exp\left(-C_1\frac{|x-y|^2}{T-s}\right)dy\\
        & \leq \frac{C_0 \|f\|_\infty}{(T-s)^{(d+|\alpha|)/2}} \exp\left(-\frac{C_1\varepsilon^2}{T-s}\right) \int_{\R^d} \exp\left(-\frac{C_1}{2}\frac{|x-y|^2}{T-s}\right) dy\\
        &\quad= \frac{(4\pi/C_1)^{d/2}C_0 \|f\|_\infty}{(T-s)^{|\alpha|/2}} \exp\left(-\frac{C_1\varepsilon^2}{T-s}\right).
    \end{align*}
    In the middle step, we used the fact that $|x-y|\geq \dist(\supp(f),x)$ for $y\in\supp(f)$. The first result follows from using the fact that $\sup_{r>0}r^{-m}\exp(-A/r)=(m/e)^m A^{-m}$.
    
    {The second assertion, under \ref{assumption_domain} and \ref{assumption_f_}, is a direct consequence of \cite[Theorem 5.14]{lieberman}.}
\end{proof}

\begin{lemma}\label{lemma 4.1}
    (\cite[Lemma 4.1]{Gobet2000}) Let $(X_t,\,t\geq0)$ be the Brownian motion as in \eqref{transition}. Let $s$ and $s^\prime$ be two times such that $0\leq s^\prime-s\leq\Delta$, $\Delta>0$. Then, 
    for $C=\frac{\Sigma_-}{8d^2}$ and for all $a>0$ it holds that 
    \begin{align*}
        \pr_x\left(\sup_{t\in[s,s^\prime]}|X_s-X_t|\geq a\right) \leq 2d\exp\left(4C|\mu|^2\Delta-Ca^2/\Delta\right).
    \end{align*}
\end{lemma}
\begin{proof}
    We split the event $\{|X_s-X_t|\geq a\}$ as follows
    \begin{align*}
        \{\sup_{t \in [s,s^\prime]}|X_s-X_t|\geq a\}\subset\left\{\sup_{t \in [s,s^\prime]}\left|\int_s^t \mu\,du\right|\geq a/2\right\}\cup\left\{\sup_{t \in [s,s^\prime]}\left|\int_s^t \sigma\,dW_u\right|\geq a/2\right\}.
    \end{align*}
    Note that, if $a/2\le |\mu|\Delta$, then for any $c>0$, it holds that
    \begin{align*}
        \pr_x\left(\sup_{t\in[s,s^\prime]}|X_s-X_t|\geq a\right) &\leq 1
        \leq 2d\exp(c4|\mu|^2\Delta-ca^2/\Delta).
    \end{align*}
    Otherwise, for $C=\frac{\Sigma_-}{8d^2}$, we have
    \begin{align}
        &\pr_x\left(\sup_{t\in[s,s^\prime]}|X_s-X_t|\geq a\right) = \pr_x\left(\sup_{t\in[s,s^\prime]}\left|\int_s^t \sigma\,dW_u\right|\geq a/2\right)\notag\\
        &\quad\le \pr_x\left(\sup_{t\in[s,s^\prime]}|\sigma\,(W_t-W_s)|\geq a/2\right)\le \pr_x\left(\sup_{t\in[s,s^\prime]}|W_t-W_s|\geq \frac{a}{2\|\sigma\|_2}\right)\label{1423-a}\\
        &\quad\le 2d\pr_x\left(\sup_{t\in[0,\Delta]}W^1_t\geq \frac{a\sqrt{\Sigma_-}}{2d}\right)
        \leq 2d\exp(-C a^2/\Delta).\label{1423-b}
    \end{align}
    Here, in the line \eqref{1423-a} we used $|\sigma x|\le \|\sigma\|_2|x|$, where $\|\sigma\|_2$ denotes the $L^2$ norm of a matrix (i.e. the spectral norm) and for which in this case it holds $\|\sigma\|_2=\sqrt{1/\Sigma_-}$, with $\Sigma_-$ being the smallest eigenvalue of $\Sigma^{-1}$. Further, in the line \eqref{1423-b}, $W^1$ stands for the standard 1-dimensional Brownain motion, and the last inequality comes from the classical
    Bernstein's inequality, see \cite[p. 153-154]{revuz1999}.
\end{proof}

\begin{lemma}\label{lemma 5.1}
    Under \ref{assumption_domain}, there is a positive constant $C=C(d,D,\Sigma,\mu)$ such that for all $(u,z)\in(0,h]\times\overbar{D}$
    \begin{align*}
        \pr_z(T_D<u)\leq C\,\pr_z(Y_u\notin D).
    \end{align*}
\end{lemma}
\begin{proof}
    Since
    \begin{align*}
        \pr_z(T_D<u) = \ex_z[\1_{T_D<u}\pr_z(Y_u\notin D|\mathcal{F}_{T_D})] + \ex_z[\1_{T_D<u}\pr_z(Y_u\in D|\mathcal{F}_{T_D})],
    \end{align*}
   it is enough to show that, for a positive constant $C$ independent of $\mathcal{F}_{T_D}$ and $u$, we have
    \begin{align}\label{ineq}
        \pr(Y_u\notin D|\mathcal{F}_{T_D}) \geq 1/C,\quad \text{on $\{T_D<u\}$}.
    \end{align}
    Indeed, this would imply that
    \begin{align*}
        \pr_z(Y_u\in D|\mathcal{F}_{T_D}) = \pr_z(Y_u\notin D|\mathcal{F}_{T_D}) \frac{(1-\pr_z(Y_u\notin D|\mathcal{F}_{T_D}))}{\pr_z(Y_u\notin D|\mathcal{F}_{T_D})}\leq C\,\pr_z(Y_u\notin D|\mathcal{F}_{T_D}).
    \end{align*}
    
    The strategy for obtaining \eqref{ineq} is the same as the one used in \cite[Lemma 5.1]{Gobet2000}, which, in turn, is based on \cite[page 250]{karatzas}. Since the domain $D$ is of class $C^{3+\alpha}$ and its boundary is compact, $D$ satisfies both the uniform exterior sphere condition and Zaremba's (exterior) cone condition. In particular, there exists an angle $\theta>0$ and a radius $R_\theta>0$ (both dependend only on $D$) such that for all $s\in \partial D$ and the cone $K(s,-n(s),\theta)\coloneqq\{y\in\R^d:\,-(y-s)\cdot n(s)\geq |y-s|\cos{\theta}\}$ (i.e. the cone centered at $s$, in the direction of the normal derivative $-n(s)$ and the angle $\theta$) it holds that $K(s,-n(s),\theta)\cap B(s,R_\theta)\subset D^c$. Therefore,
    \begin{align*}
        \pr_z(Y_u\notin D|\mathcal{F}_{T_D}) \geq \pr_z(Y_u\in K\cap B|\mathcal{F}_{T_D}) = \int_{K\cap B} P^\mu_{u-T_D}(Y_{T_D},y)dy,
    \end{align*}
    where $K\cap B = K(Y_{T_D},-n(Y_{T_D}),\theta)\cap B(Y_{T_D},R_{\theta})$. The integrand above can be further simplified by Lemma \ref{lemma bounds transition density}, so we have
    \begin{align}
        \pr_z(Y_u\notin D|\mathcal{F}_{T_D}) \geq c(d,D,\Sigma,\mu)\int_{K\cap B} P^0_{u-T_D}(Y_{T_D},y)dy.
    \end{align}

    Using the change of variables $\sqrt{u-T_D}z=y-Y_{T_D}$, the domain $K\cap B$ becomes $\{z\in \R^d:\|z\|\leq R_{\theta}/\sqrt{u-T_D},\,z\cdot(-n)\geq\|z\|\cos \theta\}$. So
    \begin{align*}
        \pr_x(Y_u\notin D|\mathcal{F}_{T_D}) &\geq c(d,\Sigma,\mu)\int_{B(0,R_{\theta}/\sqrt{u-T_D})} \1_{z\cdot(-n(Y_{T_D}))\geq\|z\|\cos \theta} \exp\left\{-\frac{\Sigma_+}{2}\|z\|^2\right\} dz,
    \end{align*}
    where $\Sigma_+$ comes from \eqref{eq: parabolic}. Using polar coordinates and $u-T_D\leq h\leq 1$, we obtain
    \begin{align*}  
        \pr_x(Y_u\notin D|\mathcal{F}_{T_D}) \geq C(d,D,\Sigma,\mu),\quad \text{on $\{T_D<u\}$}.
    \end{align*}
\end{proof}

\begin{lemma}\label{lemma: L_bound}
    Let $S$ be a subordinator with the Laplace exponent~\eqref{eq:bernstein}. Then for all $k\in\N$ it holds that
    \begin{align*}
        {\ex L_t^k \leq e\Gamma(1+k)/\phi^k(1/t),\quad t>0.}
    \end{align*}
\end{lemma}
\begin{proof}
    It is easy to see that since $\mathds{1}_{\{S_s\leq t\}}
\leq
e^{c(t-S_s)}$ for any $c>0$, we obtain 
    \begin{align*}
        \ex L_t^k &= \int_0^\infty ks^{k-1} \pr(L_t>s)ds = \int_0^\infty ks^{k-1} \pr(S_s \leq t)ds\\
        &\leq \int_0^\infty ks^{k-1} e^{ct}\ex e^{-cS_s} ds = e^{ct}k\int_0^\infty s^{k-1} e^{-s\phi(c)} ds.
    \end{align*}
    The change of variables $y=s\phi(c)$ yields
    \begin{align*}
        \int_0^\infty s^{k-1} e^{-s\phi(c)} ds = \frac{1}{\phi^k(c)} \int_0^\infty y^{k-1} e^{-y} dy = \frac{\Gamma(k)}{\phi^k(c)}.
    \end{align*}
    The claim follows by using the identity $x\Gamma(x)=\Gamma(1+x)$ and choosing $c=1/t$.
\end{proof}

\section{Auxiliary results for the numerical examples in Section \ref{s:examples}}
\subsection{Coefficients in Example \ref{ex:1}}\label{ap:B}
For the computer implementation of the series expansion \eqref{Ex.Sol.} it is needed to calculate
\begin{equation}\label{B1}
    c_n = \frac{2}{R^2}\frac{1}{J_1^2(j_{0,n})} \int_0^R yf(y)J_0(\lambda_n y)\,dy,\quad\lambda_n=j_{0,n}/R.
\end{equation}
Here we do it for $f(y)=(1-y^2/R^2)^m$, where $m>0$. Recall the definition of the Bessel functions $J$ of order $\alpha$
\begin{align}
    J_\alpha(x) = \sum_{k=0}^\infty \frac{(-1)^k}{k! \Gamma(k+\alpha+1)} \left(\frac{x}{2}\right)^{2k+\alpha}.
\end{align}
The integral \eqref{B1} becomes
\begin{equation}
    I = \int_0^R yf(y) J_0(\lambda_n y)\,dy = \sum_{k=0}^\infty \frac{(-1)^k}{(k!)^2} (\lambda_n /2)^{2k} \int_0^R y^{2k+1}(1-y^2/R^2)^m dy.
\end{equation}
By using the change of variable $t=y^2/R^2$, we have
\begin{equation}
    \begin{aligned}
        I^\prime &= \int_0^R y^{2k+1}(1-y^2/R^2)^m dy = \frac{R^{2k+2}}{2} \int_0^1 t^k (1-t)^m dt\\
        &= \frac{R^{2k+2}}{2} B(k+1,m+1) = \frac{R^{2k+2}}{2} \frac{\Gamma(k+1)\Gamma(m+1)}{\Gamma(k+1+m+1)}.
    \end{aligned} 
\end{equation}

Then,
\begin{equation}
    \begin{aligned}
        I &= \sum_{k=0}^\infty \frac{(-1)^k}{(k!)^2} (\lambda_n /2)^{2k} \frac{R^{2k+2}}{2} \frac{\Gamma(k+1)\Gamma(m+1)}{\Gamma(k+1+m+1)}\\
        &= m!\frac{R^2}{2} \sum_{k=0}^\infty \frac{(-1)^k}{k!\Gamma(k+m+1+1)} \left(\frac{\lambda_n R}{2}\right)^{2k}\\
        &= m!\frac{R^2}{2} \left(\frac{2}{\lambda_n R}\right)^{m+1} J_{m+1}(\lambda_n R),
    \end{aligned}
\end{equation}
and finally,
\begin{equation}
    c_n = \frac{m!}{J_1^2(j_{0,n})} \left(\frac{2}{j_{0,n}}\right)^{m+1} J_{m+1}(j_{0,n}).
\end{equation}

\subsection{Obtaining the eigenfunction $f_d$ in Example \ref{ex:high-dimensional-shell}}\label{ap:b2}

As in Example \ref{ex:high-dimensional-shell}, let $d\geq2$, put $\nu=d/2-1$, fix $\rho\in(0,1)$ and $\kappa>0$, and define
\begin{equation*}
	A_d\coloneqq\operatorname{diag}(a_1,\ldots,a_d),
	\qquad
	a_i\coloneqq
	\begin{cases}
		1, & i \text{ odd},\\
		1/2, & i \text{ even}.
	\end{cases}
\end{equation*}
Recall that $\mathcal Gg(x)=\sum_{i,j=1}^d(A_dA_d^T)_{ij}\partial^2_{x_ix_j}g(x)$ is the generator of interest.

Consider first the radial Dirichlet eigenvalue problem for the Laplacian
on the normalized spherical shell
\[
\mathcal A_\rho
\coloneqq \{z\in\R^d:\rho<|z|<1\},
\qquad 0<\rho<1.
\]
Writing a radial eigenfunction as $\phi(z)=R(|z|)$, the eigenpair problem 
$\Delta\phi=-q^2\phi$ becomes
\begin{equation}\label{eq:radial-SL}
	\big(r^{d-1}R'(r)\big)'
	=-q^2r^{d-1}R(r),
	\qquad
	R(\rho)=R(1)=0.
\end{equation}
Since $\rho>0$, this is a regular Sturm--Liouville problem, and
by the spectral theorem for regular Sturm--Liouville problems its eigenvalues are simple, positive, and may be ordered as a sequence tending
to $+\infty$, that is
\[
0<q_{\nu,1}^2<q_{\nu,2}^2<\cdots\nearrow\infty.
\]
For details, see, e.g., \cite[Section 5.4]{Teschl2012}.
Dividing \eqref{eq:radial-SL} by $r^{d-1}$, setting
\(R(r)=r^{-\nu}v(qr),
\)
and writing $s=qr$, a direct calculation yields
\[
s^2v''(s)+sv'(s)+(s^2-\nu^2)v(s)=0.
\]
This is Bessel's equation of order $\nu$ and therefore its solutions on $(\rho,1)$ are linear combinations of the Bessel functions of the first kind
$J_\nu(qr)$ and of the second kind $Y_\nu(qr)$, see
\cite[Section 10.2]{olver2010nist}.
Therefore, for a fixed $q>0$, a non-trivial solution satisfying the inner Dirichlet condition
$R(\rho)=0$ is, up to a multiplicative constant,
\begin{align}
	R_q(r)
	=
	r^{-\nu}
	\Big(
	Y_\nu(\rho q)J_\nu(qr)
	-
	J_\nu(\rho q)Y_\nu(qr)
	\Big).
	\label{2210}
\end{align}
Then, the outer Dirichlet condition $R_q(1)=0$
holds if and only if
\[
\mathcal F_\nu(q)
\coloneqq
Y_\nu(\rho q)J_\nu(q)
-
J_\nu(\rho q)Y_\nu(q)
=0.
\]
In other words,
\(
q^2 \text{ is an eigenvalue of \eqref{eq:radial-SL}}
\) if and only if \(
\mathcal F_\nu(q)=0
\), and those are, therefore, $0<q_{\nu,1}<q_{\nu,2}<\cdots$.

We are able now to precisely define the domain $D$ for the Dirichlet problem for $\mathcal G$. Set 
\begin{equation*}
	R_1\coloneqq\frac{q_{\nu,1}}{\kappa},
	\qquad R_0\coloneqq\rho R_1=\rho \frac{q_{\nu,1}}{\kappa},
	\qquad
	D_{A,d}\coloneqq
	\left\{x\in\R^d:\ R_0<|A_d^{-1}x|<R_1\right\}.
\end{equation*}
The domain $D_{A,d}$ is a smooth anisotropic shell with two boundary components, and the linear change of variables $z=A_d^{-1}x$ maps the Dirichlet problem for $\mathcal G$ in $D_{A,d}$ into the problem for the standard Laplacian in the spherical shell $R_0<|z|<R_1$. 

To construct the radial Dirichlet eigenfunction for $\mathcal G$ in $D_{A,d}$, define
\begin{equation}\label{eq:highdim-radial-modes}
	\Psi_1(s)
	\coloneqq R_{q_{\nu,1}} (s) =
	s^{-\nu}\left[
	Y_\nu(\rho q_{\nu,1})J_\nu(q_{\nu,1}s)
	-J_\nu(\rho q_{\nu,1})Y_\nu(q_{\nu,1}s)
	\right], \quad s\in [\rho,1],
\end{equation}
so that $\Psi_1$ satisfies \eqref{eq:radial-SL} with $q=q_{\nu,1}$. The corresponding eigenvalue for $\mathcal{G}$ in $D_{A,d}$ is then
\(
	\lambda_1=\left(\frac{q_{\nu,1}}{R_1}\right)^2
	=\kappa^2.
\)
Let $M_1\coloneqq\max_{\rho\leq s\leq1}|\Psi_1(s)|$, choose the sign $\epsilon_1\in\{-1,1\}$ so that $\epsilon_1\Psi_1$ is strictly positive in $(\rho,1)$, and let $s_\star$ be the point of maximum of $\Psi_1$. Put
\begin{equation}\label{def:f_d}
	f_d(x)\coloneqq
	\epsilon_1\frac{\Psi_1(|A_d^{-1}x|/R_1)}{M_1},\quad x\in D_{A,d},
	\qquad
	x_\star\coloneqq A_d(R_1s_\star e_1).
\end{equation}
By construction, $f_d$ is the first radial Dirichlet eigenfunction of
$\mathcal{G}$, normalized so that its maximum equals one. More
precisely,
\[
\mathcal{G}f_d=-\kappa^2 f_d \quad \text{in }D_{A,d},
\quad
f_d=0 \quad \text{on }\partial D_{A,d},
\]
and thus $\mathcal{G}f_d=0$ on
$\partial D_{A,d}$. Moreover, $f_d(x_\star)=1$.

The solution to $\partial_T^\alpha u_d(T,x)=\mathcal{G}u_d(T,x)$ is therefore
\[
	u_d(T,x)
	=
	\mathds E_x
	\left[
	f_d(X_{L_T})
	\mathbf 1_{\{L_T<T_{D_{A,d}}\}}
	\right]=
	\mathds E
	\left[
	P_{L_T}^D f_d(x)
	\right]=
	f_d(x)\,
	\mathds E
	\left[
	e^{-\kappa^2 L_T}
	\right].
\]
Here, $(P_s^D)_{s\geq 0}$ denotes the killed semigroup associated with
$\mathcal{G}$ in $D_{A,d}$; equivalently, the semigroup of the
process $X$ killed upon exiting $D_{A,d}$. We also used the fact that $f_d$ is a Dirichlet
eigenfunction,  so $P^D_s f_d(x)=e^{-\kappa^2 s}f_d(x)$.

The remaining auxiliary task required by Example \ref{ex:high-dimensional-shell} is to compute the theoretical variance $\sigma^2(T,x_\star)$ needed for the application of Theorem  \ref{CLT}. That is, denoting
\[
Z_T
=
f_d(X_{L_T})
\1_{\{L_T<T_{D_{A,d}}\}},
\]
we need to compute $\sigma^2(T,x_\star)=\ex_{x_\star}Z_T^2-(\ex_{x_\star}Z_T)^2=\ex_{x_\star}Z_T^2-(u_d(T,x_\star))^2$.

Note
\[
\ex_{x_\star}[Z_T^2]
=\ex_{x_\star}\left[f_d(X_{L_T})^2
\1_{\{L_T<T_{D_{A,d}}\}}\right]=
\ex\!\left[
P_{L_T}^D(f_d^2)(x_\star)
\right].
\]
We employ the eigenfunction expansion of $P_{s}^D(f_d^2)$ to calculate the variance.

Since $f_d$ is radial in the transformed variable $A_d^{-1}x$, so is $f_d^2$.
Introduce the radial inner product
\(
\langle g,h\rangle_\nu
\coloneqq
\int_\rho^1 g(s)h(s)s^{d-1}\,ds,
\)
and let
\(
\eta_n
\coloneqq
\frac{\Psi_n}
{\sqrt{\langle\Psi_n,\Psi_n\rangle_\nu}}
\)
be the normalized radial Dirichlet eigenfunctions, where $\Psi_n$ are given by the same expression as in \eqref{eq:highdim-radial-modes} with $q_{\nu,n}$ replacing $q_{\nu,1}$. The corresponding eigenvalues are $\lambda_n=\left(\frac{q_{\nu,n}}{R_1}\right)^2$. Writing
\[
\bar f_d(s)
=
\frac{\epsilon_1\Psi_1(s)}{M_1},
\qquad
\bar f_d(s)^2
=
\sum_{n=1}^\infty b_n\eta_n(s),
\qquad
b_n
=
\langle \bar f_d^2,\eta_n\rangle_\nu,
\]
we have
\begin{equation}\label{eq:highdim-theoretical-variance}
	\ex_{x_\star}[Z_T^2]
	=
	\sum_{n=1}^\infty
	b_n
	E_{\alpha}(-\lambda_n{T}^\alpha)
	\eta_n(s_\star).
\end{equation}
It is important to note that every quantity on the right-hand side is determined by the radial eigenpairs and one-dimensional integrals of known functions. In particular, no sample variance is required, and we can numerically evaluate \eqref{eq:highdim-theoretical-variance} by truncating the infinite series.
Let $\sigma_d^{(K)}(T,x_\star)$ denote the
corresponding approximation obtained by cutting after the first $K$ modes. We use $K=80$ in the computations and as a numerical check,
we recompute the standard deviation with $K=160$ at the $16$ equally
spaced time points
\(
T_j\in[0.0025,0.25]
\)
used in Subfigure~\ref{fig:highdim-ci}. We obtain
\begin{align}\label{variance}
	\max_{1\le j\le16}
	\left|
	\sigma_d^{(160)}(T_j,x_\star)
	-
	\sigma_d^{(80)}(T_j,x_\star)
	\right|
	\approx 8.15\times10^{-9}.
\end{align}
The implemented Python code for these calculations can be found in \cite{Python}.

	\bibliographystyle{abbrv}
	\bibliography{References}

		\bigskip
	
\noindent	{\bf Ivan Bio\v{c}i\'c}
	
\noindent	Department of Mathematics, Faculty of Science, University of Zagreb, Zagreb, Croatia,
	
\noindent	Department of Mathematics “Giuseppe Peano”, University of Turin, Turin, Italy,
	
\noindent	Email: \texttt{ivan.biocic@unito.it}, \texttt{ivan.biocic@math.hr}
	
	\bigskip
	
\noindent	{\bf Daniel E. Cedeño-Girón}
	
\noindent	Department of Mathematics “Giuseppe Peano”, University of Turin, Turin, Italy,
	
\noindent	Email: \texttt{danieleduardo.cedenogiron@unito.it}

	\bigskip
	
\noindent	{\bf Aleksandar Mijatovi\'{c}}
	
\noindent	Department of Statistics, University of Warwick, UK,
	
\noindent	Email: \texttt{a.mijatovic@warwick.ac.uk}
	
	\bigskip
\noindent	{\bf Bruno Toaldo}
	
\noindent	Department of Mathematics “Giuseppe Peano”, University of Turin, Turin, Italy,
	
\noindent	Email: \texttt{bruno.toaldo@unito.it}

\end{document}